\documentclass{article}

\usepackage{proof}
\usepackage{mathtools}
\usepackage{bussproofs}
\usepackage{multicol,stmaryrd,amssymb}
\DeclareFontFamily{U}{MnSymbolC}{}
\DeclareSymbolFont{MnSyC}{U}{MnSymbolC}{m}{n}
\DeclareMathSymbol{\boxdot}{\mathbin}{MnSyC}{"76}
\DeclareMathSymbol{\diamonddot}{\mathbin}{MnSyC}{"7E}
\DeclareFontShape{U}{MnSymbolC}{m}{n}{
    <-6>  MnSymbolC5
   <6-7>  MnSymbolC6
   <7-8>  MnSymbolC7
   <8-9>  MnSymbolC8
   <9-10> MnSymbolC9
  <10-12> MnSymbolC10
  <12->   MnSymbolC12}{}

\usepackage{enumerate}
\usepackage{hyperref}
\usepackage{cleveref}
\usepackage{authblk}
\usepackage{url}

\usepackage{tikz} 
\usetikzlibrary{arrows.meta, positioning, fit} 

\newcommand{\langu}{\mathcal L_\Diamond}

\newcommand{\langfull}{\mathcal L_{hv}}

\newcommand{\hbx}{\Box_h} %
\newcommand{\vbx}{\Box_v} %
\newcommand{\hd}{\Diamond_h} 
\newcommand{\vd}{\Diamond_v} 
\newcommand{\hr}{\prec_h}
\newcommand{\vr}{\prec_v}
\newcommand{\rh}{\succ_h}

\newcommand{\ud}{\Diamond_h} 
\newcommand{\ld}{\Diamond_v} 

\usepackage{yfonts,array}
\usepackage{amsmath,amssymb,amsthm,units,stmaryrd}

\usepackage{upgreek,xcolor,bm,tikz,stackrel}
\usepackage[inline]{enumitem}
\usepackage{graphicx}
\usepackage[all]{xy}
\usepackage{tikz-cd}
\usetikzlibrary{arrows}

\newcommand{\ignore}[1]{}

\newcommand{\val}[1]{\lb #1 \rb}

\newcommand{\peq}{\preccurlyeq}

\newcommand{\sofia}[1]{{\color{blue}\footnote{\color{blue} SOFIA: #1}}}

\def\lb{\left\llbracket}
\def\rb{\right\rrbracket}
\def\<{\left (}

\def\>{\right )}
\def\({\left (}
\def\){\right )}

\DeclareSymbolFont{AMSb}{U}{msb}{m}{n}
\DeclareMathSymbol{\N}{\mathbin}{AMSb}{"4E}
\DeclareMathSymbol{\Z}{\mathbin}{AMSb}{"5A}
\DeclareMathSymbol{\R}{\mathbin}{AMSb}{"52}
\DeclareMathSymbol{\Q}{\mathbin}{AMSb}{"51}
\DeclareMathSymbol{\I}{\mathbin}{AMSb}{"49}
\DeclareMathSymbol{\C}{\mathbin}{AMSb}{"43}

\newtheorem{thm}{Theorem}[section]
\newtheorem{proposition}[thm]{Proposition}
\newtheorem{defn}[thm]{Definition}
\newtheorem{lem}[thm]{Lemma}
\newtheorem{prop}[thm]{Proposition}
\newtheorem{corollary}[thm]{Corollary}

\newtheorem{remark}[thm]{Remark}
\newtheorem{question}[thm]{Question}

\newtheoremstyle{claimstyle}
  {}{}{\itshape}{}{\normalfont}{.}{.5em}{}

\theoremstyle{claimstyle}
\newtheorem{claim}{Claim}
\newenvironment{proofof}[1]
  {\begin{proof}[Proof of #1]}
  {\end{proof}}
  
\usepackage{amsmath}

\usepackage[T1]{fontenc}
\newcommand*{\textcal}[1]{%
  \textit{\fontfamily{qzc}\selectfont#1}%
}

\input{tikzit.sty}

\begin{document}

\title{Completeness and Incompleteness for Expanding Gödel-Löb Logics}

\author[1]{Somayeh Chopoghloo\thanks{\url{s.chopoghloo@ipm.ir}}}
\author[2]{David Fernández-Duque\thanks{\url{fernandez-duque@ub.edu}}}
\author[2,3]{Joost J.~Joosten\thanks{\url{jjoosten@ub.edu}}}
\author[2]{Sofía Santiago-Fernández\thanks{\url{sofia.santiago@ub.edu}}}

\affil[1]{School of Mathematics, Institute for Research in Fundamental Sciences (IPM), Tehran, Iran}
\affil[2]{Department of Philosophy, University of Barcelona, Barcelona, Spain}
\affil[3]{Centre de Recerca Matemàtica, Bellaterra (Barcelona), Spain}

\date{}
\maketitle

\begin{abstract}
Expanding products of modal logics are bimodal logics obtained from the combination of a `horizontal component' logic and a `vertical component' logic, lying between the fusion and the Cartesian product of the two logics.
Gabelaia et al.~showed that expanding products are often decidable when the first component is Noetherian, although their methods are semantical and do not yield complete axiomatisations.
They do, however, propose a candidate, dubbed the expanding commutator of the two logics and known to be complete in many `non-Noetherian' cases.
In this paper, we consider various expanding products of modal logics whose vertical component is $\sf GL$.
We show that the standard axiomatisation is complete when the horizontal component is either $ {\sf K4}$ or $ {\sf GL} $, but incomplete when it is  ${\sf Grz}$ or any logic between ${\sf K4.3}$ and ${\sf Grz.3}$, thus yielding a partial solution to a question posed by Gabelaia et al.~more than two decades ago.
\end{abstract}

\section{Introduction}

Modal logics used in applications often involve more than one modality, for example, to represent the knowledge of several agents or different spatial dimensions (see, e.g., \cite{Aiello2007Spatial, mdml, Ditmarsch2007DEL}).
This raises the question of how the computational complexity of such multimodal logics relates to that of the unimodal fragments, but the answer to this question depends largely on how the individual modalities interact.
On one extreme, fusions of modal logics combine the component logics with essentially no interaction and tend to have similar complexity as their components \cite{Halpern:1992:guideComplexityModalLogics}, whereas products of modal logics, where the individual relations commute with each other, very easily become undecidable or even non-axiomatisable when the component logics are transitive~\cite{mdml, Gabelaia2005}.

Somewhere in the middle, products of modal logics with expanding domains, introduced by Kurucz and Zakharyaschev~\cite{kurucz2003note}, provide one such intermediate framework. They preserve part of the structure of ordinary products while relaxing the strong constant-domain assumption underlying standard product semantics. This weakening of the standard product construction can be naturally motivated by, among others, temporal description logics \cite{lutz2008temporal, sturm2002tableau} and dynamic topological logic \cite{kremer2005dynamic, konev2006dynamic, fernandez2011dynamic}. In the former setting, it is often natural to assume that each time point is associated with a representation of the world which may grow over time but cannot shrink. In dynamic topological logic, constant domains are related to dynamical systems where the function describing the movement of points in a topological space is a homeomorphism. Meanwhile, expanding domains are related to the more relaxed requirement that the function be continuous, thereby allowing for a broader class of dynamical behaviours.

Formally, products with expanding domains combine a `horizontal' modal logic ${\sf L}_h$ with a `vertical' logic ${\sf L}_v$ to obtain a new logic $({\sf L}_h\times{\sf L}_v)^{\sf e}$, defined semantically as the logic of those frames consisting of an ${\sf L}_h$-frame $(W,R_h)$ where each $x\in W$ is assigned an ${\sf L}_v$-frame $f(x) =  (W^x,R^x _v)$ in such a way that $x R_h y$ implies that $f(x)$ is a subframe of $f(y)$. An example of these frames is shown in Figure~\ref{fig0}. When the accessibility relation is transitive, we will write $\prec$ rather than $R$, so that $R_h$ becomes $\hr$ and $R^x _v $ becomes $\vr^x$; as we will work almost exclusively with transitive frames, this will be our preferred notation.

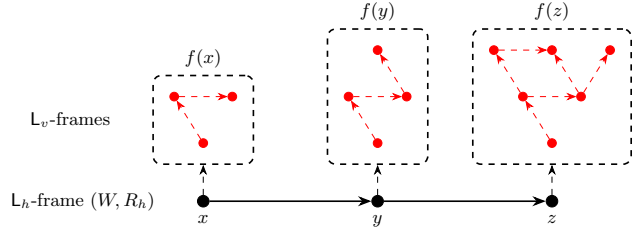
\begin{figure}[t]
\centering
\resizebox{0.7\textwidth}{!}{
\begin{tikzpicture}[
    >=Stealth,
    node distance=2.5cm,
    every node/.style={font=\small},
    world/.style={
        circle,
        draw,
        fill=black,
        inner sep=2pt
    },
    vworld/.style={
        circle,
        draw=red,
        fill=red,
        inner sep=1.5pt
    },
    framebox/.style={
        draw,
        rounded corners,
        thick,
        inner sep=8pt
    }
]
\node[vworld] (a1) at (-0.5,1.5) {};
\node[vworld] (a2) at (0.5,1.5) {};
\node[vworld] (a3) at (0,0.7) {};

\draw[->,dashed,red] (a1)--(a2);
\draw[->,dashed,red] (a3)--(a1);

\node[framebox,dashed,fit=(a1)(a2)(a3),label=above:$f(x)$] (F1) {};
\node[vworld] (b2) at (3,2.3) {};
\node[vworld] (b3) at (2.5,1.5) {};
\node[vworld] (b4) at (3.5,1.5) {};
\node[vworld] (b5) at (3,0.7) {};

\draw[->,dashed,red] (b3)--(b4);
\draw[->,dashed,red] (b4)--(b2);
\draw[->,dashed,red] (b5)--(b3);

\node[framebox,dashed,fit=(b2)(b3)(b4)(b5),label=above:$f(y)$] (F2) {};
\node[vworld] (c1) at (5,2.3) {};
\node[vworld] (c2) at (6,2.3) {};
\node[vworld] (c3) at (5.5,1.5) {};
\node[vworld] (c5) at (6.5,1.5) {};
\node[vworld] (c6) at (6,0.7) {};
\node[vworld] (c7) at (7,2.3) {};

\draw[->,dashed,red] (c5)--(c2);
\draw[->,dashed,red] (c3)--(c5);
\draw[->,dashed,red] (c3)--(c1);
\draw[->,dashed,red] (c1)--(c2);
\draw[->,dashed,red] (c6)--(c3);
\draw[->,dashed,red] (c5)--(c7);

\node[framebox,dashed,fit=(c1)(c2)(c3)(c5)(c6)(c7),  label=above:$f(z)$] (F3) {};
\node[world,label=below:$x$] (x1) at (0,-0.3) {};
\node[world,label=below:$y$] (x2) at (3,-0.3) {};
\node[world,label=below:$z$] (x3) at (6,-0.3) {};

\draw[->,thick] (x1)--(x2);
\draw[->,thick] (x2)--(x3);

\node[left=0.6cm of x1] {${\sf L}_h$-frame $(W, R_h)$};
\node[left=0.6cm of F1] {${\sf L}_v$-frames};
\draw[dashed,->] (x1)--(F1);
\draw[dashed,->] (x2)--(F2);
\draw[dashed,->] (x3)--(F3);
\end{tikzpicture}
}
\caption{An expanding domain frame}
\label{fig0}
\end{figure}

While expanding products retain high computational complexity, Gabelaia et al.~\cite{pml} showed that there are many cases where they are decidable despite standard products not being so, most notably when the horizontal logic is Noetherian, i.e., it extends $\sf GL$ or ${\sf Grz}$.\footnote{Recall that if $W$ is a set and  $R\subseteq W\times W$, then $R$ is {\em Noetherian} if whenever $w_0Rw_1Rw_2\ldots$ is an infinite sequence, there is $n\in\mathbb N$ such that, for all $k>n$, $w_k=w_n$. This is a stronger condition than {\em converse well-foundedness,} which states that no such sequence exists at all. ${\sf Grz}$ is the logic of transitive, reflexive, Noetherian frames, and $\sf GL$ is the logic of transitive, conversely well-founded frames. For definitions of some normal modal logics, such as $\sf GL$ and ${\sf Grz}$, see \Cref{sec:Hilbertcalculus}.} The issue is that the proof of decidability is highly model-theoretic and does not yield a complete axiomatisation. It is also worth noting that decidability relies on Kruskal's theorem~\cite{Kruskal1960}, and accordingly, Gabelaia et al.~\cite{pml} also show that these logics are not decidable in primitive recursive time.

Despite completeness being unknown in many cases, there is a natural candidate for an axiomatisation known as the {\em expanding commutator} of two logics, denoted $[{\sf L}_h , {\sf L}_v]^{\sf e}$. This commutator is known to be complete in many `non-Noetherian' cases; namely, $[{\sf L}_h,{\sf L}_v]^{\sf e}$ is complete whenever ${\sf L}_h\in \{{\sf K},{\sf K4},{\sf S4},{\sf S5}\}$ and ${\sf L}_v$ is axiomatisable by modal formulas with a universal Horn first-order translation~\cite{mdml}. However, logics such as $\sf GL$ do not have such an axiomatisation and thus Gabelaia et al.~left open a series of questions which can roughly be summarized as follows: {\em For which expanding product logics where at least one component is Noetherian, is the commutator complete?}

In this paper, we consider expanding products of modal logics whose `expanding' component (i.e., the vertical one) is the logic of provability $\sf GL$~\cite{Boolos:1993:LogicOfProvability}. We show that the standard axiomatisation is complete when the horizontal component is either $ {\sf K4}$ or $ {\sf GL}$, but incomplete when it is ${\sf Grz}$ or any logic lying between ${\sf K4.3}$ and ${\sf Grz.3}$ (along with many other logics; see Section~\ref{sec:IncompleteLogics}).

We first introduce a class of bi-relational Kripke frames, called {\em forward-confluent} (see, e.g.,~\cite{BalbianiDF21}) for which the commutator is sound. This class of frames is convenient because the canonical model of $[{\sf L}_h , {\sf L}_v]^{\sf e}$ is based on a forward-confluent frame, despite not being presented as an expanding domain frame. The confluence properties moreover allow us to apply a selection method to extract $({\sf K4} \times{\sf GL})^{\sf e}$-models from the canonical model of $[{\sf K4},{\sf GL}]^{\sf e}$, thus obtaining the completeness for this logic. Roughly speaking, the selection method builds a `horizontal' frame $(W,\prec_h)$ by successively adding new points, in such a way that when $x$ is added to $W$, a suitable `vertical' frame $f(x)$ is selected as well.

However, extracting a $({\sf GL} \times{\sf GL})^{\sf e}$-model from the respective canonical model requires an additional technique which does rely on a specific property of $\sf GL$. Namely, every finite tree-like $\sf GL $-model $\mathfrak m$ labelled by finite $\Sigma$-types (sets of formulas satisfying the expected closure conditions for negation and conjunction) admits an associated formula ${\rm Sim}(\mathfrak m)$ such that, for any ${\sf GL}$-model $\mathfrak M$ and any world $w$ of $\mathfrak M$,
\[
\mathfrak M,w \models {\rm Sim}(\mathfrak m) \quad \text{iff} \quad \text{there exists a $\Sigma$-simulation } e : \mathfrak{m} \to \mathfrak{M}
\nonumber
\]
where $e$ is an embedding sending the root of $\mathfrak m$ to $w$ and preserving all formulas in $\Sigma$. Such formulas exist over $\sf GL$ and over $\sf Grz$ but not over `non-Noetherian' logics such as $\sf K4$~\cite{Fernandez11}.

The reason that such formulas are needed in our completeness proof is that the canonical model of $[{\sf GL},{\sf GL}]^{\sf e}$ is not conversely well-founded (for either relation), but any {\em definable} non-empty set does have a maximal element, as given by the Löb axiom in the form $\Diamond\varphi\to \Diamond (\varphi\wedge\neg\Diamond\varphi)$. Roughly speaking, when extracting an expanding domain model from the canonical model, we will always choose $x$ so that it satisfies $\neg\Diamond {\rm Sim}(f(x))$, and hence the vertical frame assigned to $x$ may not repeat in any successor of $x$. This process terminates in finite time, for otherwise, Kruskal's theorem yields $x \prec_h y$ such that $f(x)$ embeds into $f(y)$, contradicting the fact that $x$ satisfies $\neg\Diamond {\rm Sim}(f(x))$, thus obtaining a $[{\sf GL},{\sf GL}]^{\sf e}$-model, as required.

Moreover, we show that $[{\sf L}_h,{\sf GL}]^{\sf e}$ is not complete whenever ${\sf L}_h = {\sf Grz}$ or it contains the $\sf{3}$ axiom. Here we once again employ our forward-confluent models, since by soundness we can use them to exhibit formulas which are not derivable, despite being valid over the class of expanding domain models.

\subsection*{Layout} 
The structure of the rest of the paper is as follows. In \Cref{sec:ExpandingProductsOfGLFrames}, we introduce the syntax and various semantics for the bimodal language, together with the corresponding notions of satisfiability and their interrelationships. In \Cref{sec:Hilbertcalculus}, we review a Hilbert-style calculus for the expanding commutator of two unimodal logics and establish several soundness results. In \Cref{sec:IncompleteLogics}, we use forward-confluent models to establish the incompleteness of some expanding products. \Cref{sec:Thecanonicalmodel} is devoted to constructing the canonical model, which already suffices to give an alternative proof of completeness for $[\mathsf{K4},\mathsf{K4}]^{\sf e}$, albeit for its forward-confluent semantics. \Cref{sec:Moments} reviews finite labelled trees~\cite{pml}, here called {\em moments,} along with some orderings between them. In \Cref{sec:Simulations}, we recall the notion of a simulation between a moment and the canonical model and establish some properties which are then used in \Cref{sec:Extract} to extract quasimodels from the canonical model, establishing the completeness of $[\mathsf{K4},\mathsf{GL}]^{\sf e}$ for its class of expanding domain models. In order to obtain completeness for $[\mathsf{GL},\mathsf{GL}]^{\sf e}$, \Cref{Sec:Char} defines the formula $\mathrm{Sim}(\mathfrak m)$ associated with the moment $\mathfrak m$, used to extract Noetherian quasimodels from the canonical model. Finally, \Cref{sec:Concludingremarks} provides some concluding remarks and discusses several open problems.

\section{Syntax and semantics}\label{sec:ExpandingProductsOfGLFrames}

Throughout the paper, we fix the language $\langfull$ as generated by countably many propositional variables in the set $\mathrm{Prop}$, $\neg$, $\wedge$, $\hd$, and $\vd$. We write $\hbx$ for $\neg \ud \neg$ and $\vbx$ for $\neg \ld \neg$. Since our base logic is classical, we will freely use other Boolean connectives, such as $(\varphi \vee \psi)$ and $( \varphi \to \psi)$ as abbreviations for $\neg (\neg \varphi \wedge \neg \psi)$ and $\neg (\varphi \wedge \neg \psi)$, respectively. We also adopt the abbreviations $\diamonddot_\alpha \varphi=\varphi\vee\Diamond_\alpha\varphi $ and $\boxdot_\alpha \varphi=\varphi\wedge \Box_\alpha\varphi$ for $\alpha \in \{h, v\}$.

\subsection{Expanding domain semantics}

Let us begin by recalling the original expanding domain semantics used by Gabelaia et al.~\cite{pml}.

\begin{defn}[\cite{pml}]\label{e-frame} 
Let $\mathcal{C}_h$ and $\mathcal{C}_v$ be two classes of unimodal frames. An {\em expanding domain frame} is a pair $(\mathfrak{F},f)$, where $\mathfrak{F}=(W,\hr)\in\mathcal{C}_h$ is the {\em horizontal frame} and $f:W\to\mathcal{C}_v$ is a function assigning to each $x\in W$ a {\em vertical frame} $f(x)=(W^x,\vr^x)\in\mathcal{C}_v$, in such a way that whenever $x \hr y$ for $x,y\in W$, the frame $f(x)$ is a subframe of $f(y)$; that is,
    \begin{enumerate}
    \item[(i)] $W^x \subseteq W^y$, and
    \item[(ii)] $u \vr^x w$ iff $u \vr^y w$ for all $u,w\in W^x$.
    \end{enumerate}
We denote the class of all such frames by $(\mathcal{C}_h \times \mathcal{C}_v)^{\sf e}$. 

A {\em valuation} on $(\mathfrak F,f)$ is a family $V=(V^x)_{x \in W}$ such that, for each $x \in W$,  $V^x$ is a valuation on the frame $f(x)=(W^x,\vr^x)$. A triple $\mathfrak{M} = (\mathfrak{F}, f, V)$, consisting of an expanding domain frame with a valuation is an {\em expanding domain model.}
\end{defn}

The satisfiability relation $\mathfrak{M}, (x, u) \Vdash\varphi$ for $\varphi \in \langfull$, $x \in W$ and $u \in W^x$ is recursively defined as
\begin{itemize}
    \item $\mathfrak{M}, (x, u) \Vdash p$ iff $u \in V^x(p)$ for $p \in \mathrm{Prop}$,
    \item the Booleans $\neg, \wedge$ are as usual,
    \item $\mathfrak{M}, (x, u) \Vdash \hd \varphi$ iff there is $y \in W$ such that $x \hr y$ and $\mathfrak{M}, (y, u) \Vdash\varphi$,
    \item $\mathfrak{M}, (x, u) \Vdash \vd \varphi$ iff there is $w \in W^x$ such that $u \vr^x w$ and $\mathfrak{M}, (x, w) \Vdash \varphi$.
\end{itemize}

A formula $\varphi \in \langfull$ is {\em satisfied in an expanding domain model} $\mathfrak{M}$ if there exist $x \in W$ and $u \in W^x$ such that $\mathfrak{M}, (x,u) \Vdash \varphi$. We say that $\varphi$ is {\em valid in such a model} $\mathfrak{M}$ if $\mathfrak{M}, (x, u) \Vdash \varphi$ for all $x \in W$ and $u \in W^x$. Finally, $\varphi$ is {\em valid in an expanding domain frame} if it is valid in every expanding domain model based on that frame.

\begin{remark}[\cite{pml}]
The formulas $\ld \ud p \to \ud \ld p$ and $\ud \vbx p \to \vbx \ud p$ are valid in all expanding domain frames. 
\end{remark}

An expanding domain frame or model inherits the properties of its components in the natural way, e.g.~if $(\mathfrak F,f)$ is an expanding domain frame with $\mathfrak F=(W,R)$, then we say that $(\mathfrak F,f)$ is {\em transitive} if $\hr$ is transitive and each $\vr^x$ with $x\in W$ is transitive; it is {\em Noetherian} if these relations are all Noetherian, and so on.
Note that the relations $\hr$ and $\vr^x$ need not be transitive, but as the notation suggests, we are mostly interested in the case where they are. However, it will be convenient to leave the door open to non-transitive relations, for example when discussing results in the literature.

We will often abuse notation and identify a logic $\sf L$ with its class of frames $\mathcal C(\mathsf L)$, so that if ${\sf L}_h$ and ${\sf L}_v$ are normal unimodal logics, we simply write $({\sf L}_h \times {\sf L}_v)^{\sf e}$ to denote $\big(\mathcal{C}({\sf L}_v) \times \mathcal{C}({\sf L}_h)\big)^{\sf e}$. Note that there is no risk of confusion, since $({\sf L}_h \times {\sf L}_v)^{\sf e}$ is a purely semantical construction and there is no other way to read this notation.\footnote{In contrast, in Section~\ref{sec:Hilbertcalculus} we recall the expanding commutator $[{\sf L}_h , {\sf L}_v]^{\sf e}$, which treats ${\sf L}_h$ and ${\sf L}_v$ as logics and not as classes of frames.}

\subsection{Embedding domain semantics}

The class of expanding domain models should be regarded as the `official' semantics for our logics, and our main results pertain to this class. However, a central theme of our strategy will be to broaden this semantics to obtain classes of frames that are easier to construct, but which can ultimately be transformed into expanding domain frames. We will begin with a rather innocuous, but technically convenient, generalisation of expanding domain semantics. As we will only apply this generalised semantics to transitive logics, we will assume transitivity when working with this class of frames.

\begin{defn}\label{defEmbModel}
Let $\mathcal{C}_h$ and $\mathcal{C}_v$ be two classes of transitive unimodal frames. An {\em embedding domain frame} (or {\em ${\sf e}$-frame}) is a triple $(\mathfrak{F},f,e)$, where $\mathfrak{F}=(W,\hr)\in\mathcal{C}_h$ is the \emph{horizontal frame}, $f:W\to\mathcal{C}_v$ is a function assigning to each $x\in W$ a {\em vertical frame} $f(x)=(W^x,\vr^x)\in\mathcal{C}_v$, and $e$ is a function assigning to each pair $x,y\in W$ with $x\hr y$ an {\em embedding} from $f(x)$ to $f(y)$, i.e.~an injective function $e_{xy} \colon W^x\to W^y$ satisfying
\begin{enumerate}[label=(\roman*)]
    \item  $u \vr^x w$ iff $e_{xy}(u) \vr^y e_{xy}(w)$ for all $u,w\in W^x$, and
    \item\label{itETrans}  if $x\hr y\hr z$ then $e_{yz} \circ e_{xy} = e_{xz}$.
\end{enumerate}

We denote the class of all such ${\sf e}$-frames by $(\mathcal{C}_h \rtimes \mathcal{C}_v)^{\sf e}$.
\end{defn}

Note that the second condition specifically applies to transitive frames and is one of the reasons that we will not define embedding models for non-transitive classes, although it should be possible to modify the definition for a non-transitive setting.
In practice, $(W,\hr)$ will usually be a tree, in which case it suffices to define $e_{xy}$ when $y$ is an immediate successor of $x$ and then use \ref{itETrans} to recursively define $e_{xy}$ in other cases.

Semantics for embedding domain models  (or ${\sf e}$-models) are defined analogously to those for expanding domain models, except that we set 
\begin{itemize}
\item $\mathfrak{M}, (x, u) \Vdash \ud \varphi$ iff there is $y \in W$ such that $x \hr y$ and $\mathfrak{M}, (y, e_{xy}(u)) \Vdash\varphi$.
\end{itemize}

We adopt the same conventions for $\sf e$-frames as we did for expanding domain frames, e.g.~an $\sf e$-frame inherits the properties (such as Noetherianness) of its relations, and we identify a logic with its class of frames when using the notation $({\sf L}_h\rtimes{\sf L}_v)^{\sf e}$.

Working with $\sf e$-frames rather than expanding domain frames will be more convenient for technical reasons, but in theory the distinction is inessential due to the following.

\begin{lem}\label{EquiEmbe&Expan} 
Let $\mathcal C_h$ and $\mathcal C_v$ be two classes of transitive frames.
\begin{enumerate}
    \item \label{1EquiEmbe&Expan} If a formula is satisfiable in an expanding domain frame $(\mathcal C_h \times \mathcal C_v )^{\sf e}$, then it is satisfiable in a $(\mathcal C_h \rtimes \mathcal C_v )^{\sf e}$-frame.
    \item \label{2EquiEmbe&Expan} If a formula is satisfiable in a $(\mathcal C_h \rtimes \mathcal C_v )^{\sf e}$-frame whose horizontal component is a tree, then it is satisfiable in an expanding domain frame in $(\mathcal C_h \times \mathcal C_v )^{\sf e}$.
\end{enumerate}
\end{lem}
\begin{proof}
\begin{enumerate}[wide, labelwidth=!, labelindent=5pt, itemsep=0pt]
\item[\ref{1EquiEmbe&Expan}.] 
Clearly, every expanding domain model can be regarded as an embedding domain model by just letting the embeddings be the inclusion maps.

\item[\ref{2EquiEmbe&Expan}.] Consider a $( \mathcal C_h \rtimes \mathcal C_v )^{\sf e}$-model $\mathfrak M = ((W,\prec_h),f,e, V)$  where $(W,\prec_h)$ is a tree. We can moreover assume that for $x,y\in W$, $x\neq y$ implies that $W^x\cap W^y =\varnothing$ by a standard disjoint union construction. 

Let $W^{\infty} = \bigcup_{x\in W}W^x$ and define ${\sim}\subseteq W^\infty\times W^\infty$ to be the least equivalence relation such that $w\sim e_{xy}(w)$ whenever $x\hr y$ and $w\in W^x$. Since $(W,\hr)$ is a tree and the embeddings satisfy $e_{yz}\circ e_{xy}=e_{xz}$, every $\sim$-equivalence class is precisely the collection of copies of a point obtained by repeatedly applying embeddings along a branch of the tree.

For each $x\in W$, let $\widehat W^x=\{[w]_\sim \mid w\in W^x\}$, and define a relation $\widehat{\vr}^x$ on $\widehat W^x$ by
\[
[u]_\sim\widehat{\vr}^x [w]_\sim \quad \text{iff}\quad u\vr^x w.
\]
Now let $\widehat f(x)=(\widehat W^x,\widehat{\vr}^x)$. It is not hard to check that  $((W,\hr),\widehat f)$ is an expanding domain frame. Moreover, if $\widehat V^x(p)=\{[w]_\sim \mid w\in V^x(p)\}$, then by induction on a formula $\varphi\in \langfull$, we have  
\[\mathfrak M,(x,w)\Vdash\varphi
\quad\text{iff}\quad \widehat{\mathfrak M},(x,[w]_\sim)\Vdash\varphi,\]
for every $x\in W$ and $w\in W^x$.
\end{enumerate}
\end{proof}

Thus, $(\mathcal C_h \times \mathcal C_v)^{\sf e}$ and $(\mathcal C_h \rtimes \mathcal C_v)^{\sf e}$ are essentially equivalent whenever the horizontal frame is a tree. The latter will be the case in all of the main constructions of the text, so, for practical purposes, they will be interchangeable for us. This explains our intentional choice to use similar notation for the two classes.
Working with $(\mathcal C_h \rtimes \mathcal C_v)^{\sf e}$ is convenient, since constructions can be a bit more flexible, so in the sequel we will work with this class as our `default' semantics, appealing to Lemma~\ref{EquiEmbe&Expan} to apply our results to `proper' expanding domain models.

\subsection{Forward-confluent semantics}\label{sec:ForwardConfluentSemantics}

Our proof of completeness cannot be a `proof by canonical model', since the canonical model for $\sf GL$ is not Noetherian.
However, it will be convenient to construct the canonical model as an `intermediate' model. The issue is that this model will not be based on an expanding or even embedding domain frame. Instead, it belongs to a wider class of bi-relational Kripke models, called {\em forward-confluent models}, which validate the expanding commutator axioms. As we will see, this semantics can be regarded as extending expanding domain semantics; however, the reader should be warned that this extension is less innocuous than that given by embedding domain models, as we will illustrate in \Cref{sec:IncompleteLogics}.

\begin{defn}\label{FrameDefs}\label{ModelDefs}
A {\em forward-confluent frame} is a triple $\mathcal F =(W,\prec_h,\prec_v)$, where $\prec_h$ and $\prec_v$ are binary relations on $W$ satisfying the properties of {\em Left Commutativity} (LC) and {\em Church-Rosser} (CR): 
\begin{itemize}
\item[(LC)] if $x \prec_v x' \prec_h y'$, then there exists $y\in W$ such that $x \prec_h y \prec_v y'$;
\item[(CR)] if $x' \succ_v x \prec_h y$, then there exists $y'\in W$ such that $x' \prec_h y' \succ_v y$.
\end{itemize}
Given classes of frames $\mathcal C_h$ and $\mathcal C_v$, the class of all forward-confluent frames where $(W,\prec_h ) \in \mathcal C_h$ and $(W,\prec_v ) \in \mathcal C_v$ is denoted by $\mathcal C_h \rtimes \mathcal C_v$. 

A {\em forward-confluent model} $\mathcal{M} = (W, \prec_h, \prec_v, \val\cdot)$ is a forward-confluent frame equipped
with a valuation $\val \cdot : \mathrm{Prop} \to \mathcal{P}(W)$.
\end{defn}

Models assign to each formula $\varphi \in \langfull$ a truth set $\val\varphi \subseteq W$ by extending $\val {\cdot}$ inductively as follows:
\begin{itemize}
    \item $\val{\neg\varphi} := W \setminus \val\varphi$,
    \item $\val{\varphi\wedge\psi} := \val\varphi \cap \val\psi$,
    \item $\val{\Diamond_\alpha \varphi} :=  \{w \in W \mid \exists u \in W \ (w \prec_\alpha u \ \text{and} \ u \in \val\varphi)\}$ for $\alpha \in \{h, v\}$. 
\end{itemize}

A formula $\varphi \in \langfull$ is {\em satisfied} in $\mathcal{M}$ if there exists some $w \in W$ such that $w \in \val{\varphi}$, in which case we write $\mathcal{M}, w \Vdash \varphi$. When the underlying model is clear from   context, we simply write $w \Vdash \varphi$. Validity in forward-confluent models and frames is defined in the usual way. In particular, for a forward-confluent frame $\mathcal F$, the notation $\mathcal F \vDash \varphi$ indicates validity of $\varphi$ on $\mathcal F$.

\begin{remark}
To aid in legibility, we use a gothic font as in e.g.~$\mathfrak M$ to denote expanding or embedding domain models and a calligraphic font as in e.g.~$\mathcal M$ to denote forward-confluent models.
\end{remark}

We adopt the conventions of expanding domain models regarding notation and terminology; in particular, $(W,\prec_h,\prec_v)$ is {\em transitive} if both $\prec_h$ and $\prec_v$ are, and, while the relations will typically be transitive, non-transitive relations will be allowed in principle. Moreover, when writing ${\sf L}_h \rtimes {\sf L}_v$, we identify each logic with its corresponding class of frames. The following is an easy observation; note that transitivity is not assumed.

\begin{proposition}\label{propLC}
Let $\mathcal F = (W, \prec_h, \prec_v)$ be an arbitrary bi-relational frame.
Using $\models$ to denote first-order validity, we have:
\begin{enumerate}
    \item $\mathcal F \models \mbox{(LC)} $ if and only if $\mathcal F \vDash \ld \ud p \to \ud \ld p$, and
    \item $\mathcal F \models \mbox{(CR)}$ if and only if $ \mathcal F \vDash \hd \vbx p \to \vbx \hd p$.
\end{enumerate}
\end{proposition}
\begin{proof}
We comment only on the left-to-right direction of the second item. So, we assume $\mathcal F \models \mbox{(CR)}$ and for arbitrary valuation, and $x\in W$ we assume $x\Vdash \ud \vbx p$. Consequently, we can find $y$ with $x\prec_h y\Vdash \vbx p$. We need to see that $x\Vdash \vbx \hd p$.
Let $x'$ be such that $x\prec_v x'$.
We know by (CR) that there is some $y'$ with $x'\prec_h y' \succ_v y$, whence $y'\Vdash p$ since $y\Vdash \vbx p$.
Thus, $x'\Vdash \hd p$ and, since $x'$ was arbitrary, $x\Vdash \vbx \hd p$.
\end{proof}

\begin{remark}
We adapt the terminology of `left commutativity' and `Church-Rosser' from Gabelaia et al.~\cite{pml}, who use them to refer to the respective formulas of Proposition~\ref{propLC}. Note moreover that $\ud \vbx p \to \vbx \ud p$ and $\vd \hbx p \to \hbx \vd p$ are equivalent, which reflects the symmetry observed in (CR).
\end{remark}

Gabelaia et al.~\cite{pml} had already shown that any expanding domain model can be transformed into a bi-relational Kripke model; as it turns out, their construction moreover yields forward-confluent models.

\begin{lem}\label{ExtoFCModel}
Let $\mathcal C_h$ and $\mathcal C_v$ be classes of Kripke frames closed under disjoint unions.
If a formula is satisfied in a $(\mathcal C_h\times \mathcal C_v)^{\sf e}$-model, then it is satisfied in a $\mathcal C_h\rtimes \mathcal C_v$-model.
\end{lem}
\begin{proof}
Given an expanding domain model $\mathfrak{M} = (\mathfrak{F}, f, V)$ where $\mathfrak{F} = (W, \hr )$, $f(x) = (W^x, \vr^x)$ for each $x \in W$ and $V=(V^x)_{x \in W}$, we define a  model $\overline{\mathfrak{M}} := (\overline{W}, \prec_h, \prec_v, \val\cdot)$, where
\begin{itemize}
    \item $\overline{W}:= \{(x,u) \mid x \in W, u \in W^x\}$;
    \item $(x,u) \prec_h (y, w)$ iff $x \hr y$ and $u = w$;
    \item $(x, u) \prec_v (y,w)$ iff $x = y$ and $u \vr^x w$;
    \item $\val p := \{(x, u) \mid u \in V^x(p)\}$ for any $p \in \mathrm{Prop}$.
\end{itemize}
We can show that $\mathfrak{M}, (x, u) \Vdash \varphi$ if and only if $\overline{\mathfrak{M}}, (x, u) \Vdash \varphi$ by an easy induction on $\varphi \in \langfull$, whose details are left to the reader. It only remains to verify that the confluence conditions are met. 
\begin{description}[wide, labelwidth=!, labelindent=5pt, itemsep=0pt]
\item[(LC)] Suppose $(x,u) \prec_v (x',u') \prec_h (y',w')$. Then $x=x'$, $u \vr^x u'$, $u'=w'$, and $x'\hr y'$. Consider $(y',u)$. Since $x=x'$ and $x' \hr y'$, we have $x \hr y'$, and therefore $(x,u) \hr (y',u)$. Moreover, since $u \vr^x u'$ and $x \hr y'$, the subframe condition implies $u \vr^{y'} u'$. As $u'=w'$, we conclude $(y',u) \vr (y',w')$.

\item[(CR)] Suppose $(x',u') \succ_v (x,u) \prec_h (y,w)$. Then $x=x'$, $u \prec_v^x u'$, $u=w$, and $x \prec_h y$. Consider $(y,u')$. Since $x \prec_h y$ and $x=x'$, we obtain $(x',u') \prec_h (y,u')$. Furthermore, from $u \prec_v^x u'$ and the subframe condition we get $u \prec_v^y u'$, which yields $(y,w) \prec_v (y,u')$ since $u=w$. 
\end{description}

It is readily verified that $(\overline W,\prec_h)$ and $(\overline W,\prec_v)$ are isomorphic to disjoint unions of frames from $\mathcal C_h$ and $\mathcal C_v$, respectively, so indeed the new model belongs to $\mathcal C_h\rtimes \mathcal C_v$.
\end{proof}

\section{Expanding commutators} \label{sec:Hilbertcalculus}

In this section, we review the expanding commutator of two modal logics, i.e.~the Hilbert-style calculi $[\mathsf L_h,\mathsf L_v]^{\sf e}$ given individual unimodal logics ${\sf L}_h$ and ${\sf L}_v$.
Before this, let us briefly recall  the unimodal logics that will be used in the text.

We formulate unimodal logics in the language with a single modality $\Diamond$, which we denote $\langu$. For our purposes, a {\em logic} is a set of formulas containing all propositional tautologies and closed under modus ponens, substitution and necessitation.
The basic modal logic $\sf K$ is then the least logic over $\langu$ containing $\Box(p \to q)\to (\Box p \to \Box q)$ and $\Diamond p \leftrightarrow \neg \Box\neg p$,\footnote{This axiom is required when $\Diamond$ is taken as a primitive.} and any logic extending $\sf K$ is a {\em normal modal logic.}

The following well-known modal logics will appear in the text.
\begin{itemize}
    \item $\sf K4$ is obtained from $\sf K$ by adding the transitivity axiom $\Diamond \Diamond p \to \Diamond p.$
    \item $\sf S4$ is obtained from $\sf K4$ by adding the reflexivity axiom $p \to \Diamond p.$
    \item $\sf Grz$ is obtained from $\sf S4$ by adding the Grzegorczyk axiom
    $$p\to \Diamond(p\wedge \Box( \Diamond p\to p)).$$
    \item $\sf GL$ is obtained by extending $\sf K4$ with the Löb axiom
    $$\Diamond p \to \Diamond( p \wedge \neg \Diamond p).$$
    \item Given a normal modal logic $\sf L$, the logic $\sf L.3$ is obtained from $\sf L$ by adding the weak connectedness axiom, also known as the $\sf 3$ axiom,
    $$(\Diamond p \wedge \Diamond q) \to  \Diamond( p \wedge \diamonddot q) \vee \Diamond(q \wedge \diamonddot p).$$
\end{itemize}

Expanding commutators of modal logics combine two individual modal logics $\mathsf L_h$ and $\mathsf L_v$, with $\hd$ replacing the modality of $\mathsf L_h$ and $\vd$ that of $\mathsf L_v$. Accordingly, say that a {\em horizontal instance} of a unimodal formula $\varphi\in\langu $ is the formula $\varphi_h$ obtained by replacing every instance of $\Diamond$ by $\hd$ and, similarly, a {\em vertical instance} of  $\varphi $ is the formula $\varphi_v$ obtained by replacing every instance of $\Diamond$ by $\vd$.
Our definition of {\em logic} can be applied uniformly to $\langfull$, where substitution can be applied to arbitrary formulas of $\langfull $ and necessitation can be applied to either modality.

\begin{defn}\label{defExpandingCommutator}
Let $\mathsf{L}_h$ and $\mathsf{L}_v$ be normal unimodal logics.
The \emph{expanding commutator} $[{\sf L}_h, {\sf L}_v]^{\sf e}$ is the least logic over $\langfull$ containing all horizontal instances of formulas of $\mathsf{L}_h$, all vertical instances of formulas of $\mathsf{L}_v$, and the expanding product axiom schemes
\begin{description}
\item[] $\Diamond_v \Diamond_h p \to \Diamond_h \Diamond_v p$ \hfill \textit{(Left commutativity)} 
\item[] $\Diamond_v \Box_h p \to \Box_h \Diamond_v p$ \hfill \textit{(Church--Rosser)}.
\end{description}
\end{defn}

If $\sf L$ is any logic (over either the unimodal or bimodal language), we say that a formula $\varphi$ is {\em derivable in $\sf L$ from a set $\Gamma$ of formulas}, and write $\Gamma \vdash_{\sf L} \varphi$, if either $\vdash_{\sf L} \varphi$ or there exist formulas $\gamma_1, \ldots, \gamma_n \in \Gamma$ such that
\[
\vdash_{\sf L} (\gamma_1 \land \cdots \land \gamma_n) \to \varphi.
\]
Here, $\vdash_{\sf L} \varphi$ means that $\varphi$ is a {\em theorem of ${\sf L}$}. We may write $\vdash$ instead of $\vdash_{\sf L}$ when $\sf L$ is clear from context.

\begin{remark} \label{RemarkK4xK4}
Let $\mathsf{L}_h, \mathsf{L}_v$ be unimodal logics, let $\varphi, \psi, \varphi_1, \dots, \varphi_n \in \langfull$. The usual reasoning is available in the expanding commutator $[{\sf L}_h, {\sf L}_v]^{\sf e}$ for both modalities given $\alpha \in \{h,v\}$.
    \begin{enumerate}
    \item \label{1RemarkK4xK4}$\vdash \Box_\alpha \varphi_1 \wedge ... \wedge \Box_\alpha \varphi_n \leftrightarrow \Box_\alpha (\varphi_1 \wedge ... \wedge \varphi_n)$.
    \item \label{2RemarkK4xK4} $\vdash \Diamond_\alpha (\varphi_1 \wedge ... \wedge \varphi_n) \to \Diamond_\alpha \varphi_1 \wedge ... \wedge \Diamond_\alpha \varphi_n$.
    \item \label{3RemarkK4xK4} If $\;\vdash \varphi \to \psi$, then $\;\vdash \Diamond_\alpha \varphi \to \Diamond_\alpha \psi$.
    \end{enumerate}
\end{remark}

As we have seen, the class of forward-confluent frames can be seen as generalising expanding domain frames; however, in many (though not all!) cases, the logics coincide, and in particular, the expanding commutator is sound for classes of forward-confluent frames.
Note that the following does not require transitivity and can be applied to any expanding commutator.

\begin{prop}\label{SoundFor-Conflu}
Let ${\sf L}_h, {\sf L}_v$ be unimodal logics. The expanding commutator $[\mathsf L_h,\mathsf L_v]^{\sf e}$ is sound for the class ${\sf L}_h\rtimes {\sf L}_v$.
\end{prop}
\begin{proof}
It suffices to show that the expanding product axioms are valid over ${\sf L}_h\rtimes {\sf L}_v$. The only interesting cases are left commutativity and the Church–Rosser, whose soundness has already been shown in Proposition~\ref{propLC}.
\end{proof}

The following is already stated for expanding domain frames by Gabelaia et al.~\cite{pml}, but it is instructive to note that it can be derived using forward-confluent frames as an intermediary.

\begin{corollary}\label{SoundFor-expan&emdedd}
Let ${\sf L}_h \in \{\sf K4, GL\}$. The expanding commutator $[{\sf L}_h,\sf GL]^{\sf e}$ is sound for $({\sf L}_h \times {\sf GL})^{\sf e}$ and $({\sf L}_h \rtimes {\sf GL})^{\sf e}$.
\end{corollary}
\begin{proof}
This follows immediately from Proposition \ref{SoundFor-Conflu} together with Lemmas \ref{ExtoFCModel} and \ref{EquiEmbe&Expan}. 
\end{proof}

\section{Incomplete logics}\label{sec:IncompleteLogics}

In this section, we will use forward-confluent models to show that some expanding commutators are incomplete for their class of expanding domain models. We begin with cases where the horizontal logic is linear.

\newcommand{\lin}[1]{\ensuremath{\text{\textcal{#1}}^{*}}}
\newcommand{\frk}[1]{\ensuremath{\text{\textcal{#1}}^{\Ydown}}}

Below, for $n\in\mathbb N$, let $\frk n $ be the irreflexive `$n$-fork' frame consisting of $n+1$ points: a root with $n$ successors that themselves do not have further successors. Let $\lin n$ be the reflexive linearly ordered frame with $n$ elements.

\begin{thm}\label{theorem:IncompletenessUsingLinearity}
Let ${\sf L}_h\supseteq \sf K4.3$ be valid on $\lin 3$ and ${\sf L}_v$ be any logic valid on $\frk 2$. Then, 
$[{\sf L}_h, {\sf L}_v]^{\sf e}$ is incomplete for $({\sf L}_h \times {\sf L}_v)^{\sf e}$.

In particular, all of $[{\sf Grz.3}, {\sf GL}]^{\sf e}$, $[{\sf K4.3}, {\sf GL}]^{\sf e}$, and $[{\sf K4.3}, {\sf K}]^{\sf e}$ are incomplete.
\end{thm}
\begin{proof}
Consider the formula
\[
\varphi := \hd p\wedge \vd\hd q \to \hd(p\wedge \vd \diamonddot_h   q) \vee \hd(\diamonddot_h p\wedge \vd    q). 
\]
We claim that $\varphi$ is valid over $(\mathsf{K4.3}\times\mathsf K)^{\sf e}$, whence over $(\mathsf L_h\times\mathsf  L_v)^{\sf e}$.

To see this, fix a $ (\mathsf{K4.3}\times\mathsf K)^{\sf e}$-model $\mathfrak M $   and assume that $ (x,u)\Vdash \hd p\wedge \vd\hd q$.
From the left conjunct, we find some $(y,u)$ with $y \succ_h x$ such that $(y,u)\Vdash p$, while the right conjunct yields $(z,w)\Vdash q$ for $x \prec_h z$ and $u \prec_v^{z} w$. Since the horizontal frame satisfies the $\sf{3}$ axiom and is transitive, we get either $y\prec_h z$ or $z=y \vee z \prec_h y$, corresponding to $\hd(p\wedge \vd \diamonddot_h q)$ and $\hd(\diamonddot_h p\wedge \vd q)$, respectively.

We now consider the forward-confluent model $\mathcal{M}_1 := (W, \prec_h, \prec_v, V)$ defined by
\begin{itemize}
    \item $W := \{0,1,2\} \times \{0,1\}$;
    \item $(x,y) \hr (x',y')$ iff $x\leq x'$ and $y=y'$;
    \item $(x,y) \vr (x',y')$ iff $(x,x') \in \{(0,0),(1,2),(2,1), (2,2)\}$ and $y<y'$;
    \item $V(p) := \{(1,0)\}$ and $V(q) := \{(1,1)\}$.  
\end{itemize}
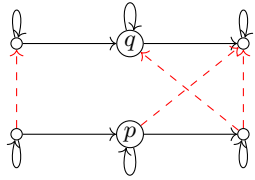
\begin{figure}[h] 
\centering
\small
\begin{tikzpicture}
	\begin{pgfonlayer}{nodelayer}
		\node [style=sroot, label=left:{}] (1) at (-1.5, 2.5) {};
		\node [style=sroot, label=above:{}] (2) at (-1.5, 3.7) {};
        \node [style=sroot, label=left:{}] (3) at (0, 3.7) {$q$};
		\node [style=sroot, label=right:{}] (4) at (0, 2.5) {$p$};
		\node [style=sroot, label=above:{}] (5) at (1.5, 3.7) {};
		\node [style=sroot, label=right:{}] (6) at (1.5, 2.5) {};

        \draw [style=edgeto, red, dashed] (1) to (2);
		\draw [style=edgeto] (2) to (3);
        \draw [style=edgeto] (3) to (5);
		\draw[style=edgeto] (1) to (4);
		\draw[style=edgeto] (4) to (6);
        \draw [style=edgeto, red, dashed] (6) to (5);
        \draw [style=edgeto, red, dashed] (6) to (3);
        \draw [style=edgeto, red, dashed] (4) to (5);
        \draw[style=edgeto] (2) edge[loop above] ();
        \draw[style=edgeto] (3) edge[loop above] ();
        \draw[style=edgeto] (5) edge[loop above] ();
        \draw[style=edgeto] (1) edge[loop below] ();
        \draw[style=edgeto] (4) edge[loop below] ();
        \draw[style=edgeto] (6) edge[loop below] ();
	\end{pgfonlayer}
\end{tikzpicture}
\caption{The model $\mathcal{M}_1$ with the $\hr$-transitive closure not displayed.}
\label{fig1}
\end{figure}

One should carefully check that the model indeed is a forward-confluent model, that is, both (LC) and (CR) are satisfied. We observe that only $\hr$-reflexivity of the points $(2,0)$ and $(2,1)$ is strictly needed to comply with (LC) and (CR). We made the other worlds $\hr$-reflexive too as to obtain a $\sf Grz.3$ model.

It is straightforward to verify that $\mathcal{M}_1, (0, 0) \not \Vdash \varphi$ and using the assumption that $\lin 3\models \mathsf L_h$ and $\frk 2\models \mathsf L_v$, the model $\mathcal M_1$ is a $\mathsf L_h\rtimes \mathsf L_v$-model. By Proposition \ref{SoundFor-Conflu}, we obtain $\not\vdash_{[{\sf L}_h, {\sf L}_v]^{\sf e}} \varphi$ for ${\sf L}_h$ and ${\sf L}_v$ as in the statement of this theorem. Since $\varphi$ holds on any  $(\mathsf L_h\times \mathsf L_v)^{\sf e}$-model, we conclude the incompleteness result. 
\end{proof}

The next theorem shows that we do not necessarily need linearity to obtain incompleteness: it can be replaced by requiring existence of strict $\hr$-maximal points.

\begin{thm}
Let $\mathsf L_h \supseteq \sf Grz$ be valid on $\lin 2$ and ${\sf L}_v$ be any logic that is valid on $\frk 2$.
Then, the logic $[\mathsf L_h , \mathsf L_v]^{\sf e}$ is incomplete for its expanding domain semantics. 

In particular, all of $[{\sf{Grz}} , {\sf K}]^{\sf e}$,
$[{\sf Grz} , {\sf K4}]^{\sf e}$, and
$[{\sf Grz} , {\sf GL}]^{\sf e}$ are incomplete for their expanding domain semantics.
\end{thm}
\begin{proof}
We will see that
\[
\psi \ := \ p\wedge \hbx (\boxdot_v  p\vee \boxdot_v \neg p )  \to \hd (p\wedge \vbx (q\to \hbx (p\to q))) 
\]
is valid over $ (\sf Grz \times K)^{\sf e}$.
To this end, we consider an arbitrary $ (\sf Grz \times K)^{\sf e}$-model $\mathfrak M$ and some $(x,x')$ with $\mathfrak{M}, (x,x')\Vdash p\wedge \hbx (\boxdot_v  p\vee \boxdot_v \neg p )$.
Since the horizontal accessibility is reflexive, we will write it as $\peq_h$ and reserve $\hr$ for its irreflexive part, i.e. $x\hr y$ if $x\peq_h y$ but $y\not\peq_h x$.

Since the horizontal frame is Noetherian and transitive, there is $y$ with $x \peq _hy$ so that 
$(y,x')\Vdash \boxdot_v p$ but, whenever $(y,x')\hr (z,x')$, we have that $(z,x')\Vdash \boxdot_v \neg p$. It is now easy to see that $(y,x')\Vdash \vbx (q\to \hbx (p\to q))$: if $x' \prec_v^yu$, by the design of $\sf e$-frames, any strict $\prec_h$ successor of $(y,u)$ cannot validate $p$ and in particular validates $p\to q$. If $(y,u)\nVdash q$, then trivially $(y,u)\Vdash q \to \hbx (p \to q)$ and if $(y,u)\Vdash q$, then $(y,u)\Vdash p\to q$ whence $(y,u)\Vdash \hbx (p\to q)$ and $(y,u)\Vdash q \to \hbx (p\to q)$. We conclude $(y,x')\Vdash p\wedge \vbx (q\to \hbx (p\to q))$ whence $(x,x') \Vdash p\wedge \hbx (\boxdot_v  p\vee \boxdot_v \neg p ) \to \hd (p\wedge \vbx (q\to \hbx (p\to q)))$ as was to be shown.

With $\mathsf L_h$ and $\mathsf L_v$ being as in the statement of the theorem, we now consider the $\mathsf L_h\rtimes \mathsf L_v$-model $\mathcal{M}_2 = (W,\peq_h,\prec_v,\val\cdot)$ as depicted in Figure \ref{fig2}.
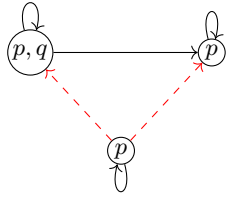
\begin{figure}[h] 
\centering
\small
\begin{tikzpicture}
	\begin{pgfonlayer}{nodelayer}
		\node [style=sroot, label=left:{}] (1) at (-1.2, 3.8) {$p, q$};
		\node [style=sroot, label=above:{}] (2) at (1.2, 3.8) {$p$};
        \node [style=sroot, label=left:{}] (3) at (0, 2.5) {$p$};

        \draw [style=edgeto] (1) to (2);
		\draw [style=edgeto, red, dashed] (3) to (1);
        \draw [style=edgeto, red, dashed] (3) to (2);
        \draw[style=edgeto] (1) edge[loop above] ();
        \draw[style=edgeto] (2) edge[loop above] ();
        \draw[style=edgeto] (3) edge[loop below] ();
	\end{pgfonlayer}
\end{tikzpicture}
\caption{The model $\mathcal{M}_2$.}
\label{fig2}
\end{figure}
More concretely, we have that 
\begin{itemize}
    \item $W :=\{0,1,2 \}$;
    \item  $1\peq_h 2$ and $i\peq_h i$ for $i\in \{0,1,2 \}$;
    \item $0\vr 1$ and $0\vr 2$;
    \item $V(p) :=\{0,1,2\}$ and  $V(q) := \{1\}$.
\end{itemize}
Using the assumption that $ \lin 2\Vdash \mathsf L_h$ and $\frk 2\Vdash \mathsf L_v$, it is quite simple to see that $\mathcal{M}_2$ is a $\mathsf L_h\rtimes \mathsf L_v$-model and that $\mathcal{M}_2, 0\not \Vdash \psi$.
\end{proof}

Thus, forward-confluent models do not in general yield the same logics as their expanding domain counterparts.
However, as we will see in the rest of the paper, there are many cases where this {\em is} the case, and in such cases, forward-confluent models will prove to be a useful tool.

\section{The canonical forward-confluent model} \label{sec:Thecanonicalmodel}

Now that we have exhibited some incomplete expanding commutators, it is time to focus on two prominent cases where these are complete: namely, $[{\sf K4},{\sf GL}]^{\sf e}$ and $[{\sf GL},{\sf GL}]^{\sf e}$. A go-to technique for proving the completeness of various modal logics is to construct their canonical model, but it is well known that the canonical model for $\sf GL$ is not Noetherian, nor is the canonical model for an expanding commutator based on an expanding domain (or even embedding domain) frame in any reasonable sense. Nevertheless, constructing it will prove very useful as an intermediate step; notably, the canonical model of an expanding commutator is based on a forward-confluent frame, our starting point for extracting $\sf e$-frames (and, hence, expanding domain frames).

As with any other (poly)modal logic, a set $\Phi$ of formulas is {\em consistent} in $[{\mathsf L_h},{\mathsf L_v}]^{\sf e}$ if $\Phi \nvdash_{[{\mathsf L_h},{\mathsf L_v}]^{\sf e}} \perp$. The set $\Phi$ is {\em maximally consistent} in $[{\mathsf L_h},{\mathsf L_v}]^{\sf e}$ if it is consistent and maximal with respect to set inclusion among all consistent sets.

\begin{defn}
Let $\mathsf L_h$, $\mathsf L_v$ be unimodal logics. The \emph{canonical model for $[\mathsf{L}_h,\mathsf{L}_v]^{\sf e}$} is the tuple
\[\mathcal M^{\rm c} := (W^{\rm c},\hr^{\rm c},\vr^{\rm c}, \val\cdot ^{\rm c}),\]
where
\begin{itemize}
    \item $W^{\rm c}$ is the set of all maximally consistent sets in $[\mathsf{L}_h,\mathsf{L}_v]^{\sf e}$;
    \item $\prec_\alpha^{\rm c}$ for $\alpha \in \{h,v\}$ is the binary relation on $W^{\rm c}$ defined by $\Phi \prec_\alpha^{\rm c} \Psi$ iff for all $\varphi \in \langfull$, $\varphi \in \Psi$ implies $\Diamond_\alpha \varphi \in \Phi$;
    \item $\val p^{\rm c} := \{\Phi \in W^{\rm c} \mid p \in \Phi \}$ for $p \in \mathrm{Prop}$.   
\end{itemize}
\end{defn}

The standard properties of canonical models can be adapted to the present bimodal setting. Since the arguments follow the usual lines, we omit the proofs.

\begin{lem} \label{BoxCanonicalModels} 
Let $\Phi \in W^{\rm c}$ and $\alpha \in \{h,v\}$. Then, $\Phi \prec_\alpha^{\rm c} \Psi$ iff for any $\varphi \in \langfull$, $\Box_\alpha \varphi \in \Phi$ implies $\varphi \in \Psi$. 
\end{lem}
\begin{proof}
For a proof, see, e.g., \cite[Lemma 4.19]{BRV01}.
\end{proof}

As usual, one of the most crucial `basic' properties of the canonical model is the witness lemma, both in its standard version and its modified version for $\sf GL$.

\begin{lem}[Witness property] \label{WitnessCanonicalModel}
Let $\Phi \in W^{\rm c}$, $\alpha \in \{h,v\}$ and $\varphi \in \langfull$. If $\Diamond_\alpha \varphi \in \Phi$, then there exists some $\Psi \in W^{\rm c}$ such that $\Phi \prec_\alpha^{\rm c} \Psi$ and $\varphi  \in \Psi$. If moreover $\mathsf L_\alpha$ extends $\mathsf{GL}$, then we may choose $\Psi$ such that $\varphi \wedge \neg \Diamond  _\alpha \varphi\in \Psi$. 
\end{lem}
\begin{proof}
The standard witness lemma is very well known (see, e.g.,  \cite[Lemma 4.20]{BRV01}). The modified version for $\sf GL$ is also well known, but we review the proof: suppose that $\mathsf L_\alpha$ extends $\sf GL$ and $\Diamond_\alpha \varphi \in \Phi$. Since $\Phi$ is maximal consistent, it is closed under derivability; but, by the Löb axiom (in the form we have stated it), from $\Diamond_\alpha\varphi$, we can infer $\Diamond_\alpha (\varphi \wedge\neg\Diamond_\alpha \varphi)$, hence the latter belongs to $\Phi$. Then, apply the standard witness lemma to find $\Psi\rh \Phi$ with $\varphi \wedge\neg\Diamond_\alpha \varphi \in \Psi $.
\end{proof}

\begin{proposition} \label{Canonicalmodel$[K4,K4]^e$}
Let $\mathsf L_h$, $\mathsf L_v$ be unimodal logics and let $\mathcal M^{\rm c}$ be the canonical model for $[ \mathsf L_h, \mathsf L_v]^{\sf e}$. Then:
\begin{enumerate}
    \item\label{itForConf} $\mathcal M^{\rm c}$ is forward-confluent.
    \item\label{itTrans} For $\alpha\in\{h,v\}$, if $\mathsf L_\alpha$ extends $\sf K4$, then $\prec_\alpha$ is transitive.
    \item\label{TruthLemma} The truth lemma holds for $\mathcal M^{\rm c}$, i.e. for any $\Phi \in W^{\rm c}$ and $\varphi \in \langfull$, we have $\varphi \in \Phi$ iff $\mathcal M^{\rm c}, \Phi \Vdash \varphi$.
\end{enumerate}
\end{proposition}
\begin{proof}
\begin{enumerate}[wide, labelwidth=!, labelindent=0pt, itemsep=5pt]
\item[\ref{itForConf}.] 
    \begin{description}[wide, labelwidth=!, labelindent=5pt, itemsep=0pt]
    \item[(LC)] Assume that $\Phi \prec_v^{\rm c} \Upsilon \prec_h^{\rm c} \Psi$. Let $\Phi^{\Box_h} := \{\varphi \mid \Box_h \varphi \in \Phi\}$ and $\Psi^{\Diamond_v} := \{\Diamond_v \psi \mid \psi \in \Psi\}$. It suffices to show that $\Theta^{-} := \Phi^{\Box_h} \cup \Psi^{\Diamond_v}$ is consistent in $[{\sf L}_h, {\sf L}_v]^{\sf e}$. We can then use Lindenbaum's lemma and Lemma \ref{BoxCanonicalModels} to conclude that there is a maximal consistent extension $\Theta \supseteq \Theta^{-}$ which satisfies $\Phi \prec_h^{\rm c} \Theta \prec_v^{\rm c} \Psi$. 

    Suppose, towards a contradiction, that $\Theta^-$ is inconsistent. Then, there exist finite subsets $\Phi_0\subseteq\Phi$ and $\Psi_0 \subseteq \Psi$ such that $\vdash \bigwedge \Psi_0^{\Diamond_v} \to \neg \bigwedge \Phi_0^{\Box_h}$. By Remark~\ref{RemarkK4xK4}\eqref{2RemarkK4xK4}, we have $\vdash \Diamond_v \bigwedge \Psi_0 \to \neg \bigwedge \Phi_0^{\Box_h}$. Then, applying Remark~\ref{RemarkK4xK4}\eqref{3RemarkK4xK4} we deduce
    \begin{equation}
    \label{EqConfl1}
    \vdash \Diamond_h \Diamond_v \bigwedge \Psi_0 \to \neg \Box_h \bigwedge \Phi_0^{\Box_h}.
    \end{equation}

    On the other hand, since $\Upsilon \prec_h^{\rm c} \Psi$, we clearly have that $\Diamond_h \bigwedge \Psi_0 \in \Upsilon$. Moreover, given that $\Phi \prec_v^{\rm c} \Upsilon$, we get $\Diamond_v \Diamond_h \bigwedge \Psi_0 \in \Phi$. Given that $\Phi$ is a $[{\sf L}_h,{\sf L}_v]^{\sf e}$-maximal consistent set, by the left-commutativity axiom we have that $\Diamond_h \Diamond_v \bigwedge \Psi_0 \in \Phi$. Therefore, from (\ref{EqConfl1}) we obtain $\neg \Box_h \bigwedge \Phi_0^{\Box_h} \in \Phi$, that is, $\Box_h \bigwedge \Phi_0^{\Box_h} \notin \Phi$. This concludes the proof by contradiction using Remark~\ref{RemarkK4xK4}\eqref{1RemarkK4xK4}.

    \item[(CR)] Assume that $\Upsilon \succ_v^{\rm c} \Phi \prec_h^{\rm c} \Theta$. Let $\Upsilon^{\Box_h} := \{\varphi \mid \Box_h \varphi \in \Upsilon\}$ and $\Theta^{\Box_v} := \{\psi \mid \Box_v \psi \in \Theta\}$. It suffices to see that $\Psi^{-} := \Upsilon^{\Box_h} \cup \Theta^{\Box_v}$ is consistent in $[{\sf L}_h, {\sf L}_v]^{\sf e}$.
    By Lindenbaum's lemma we can find a maximal consistent extension $\Psi \supseteq \Psi^{-}$ which satisfies $\Upsilon \prec_h^{\rm c} \Psi \succ_v^{\rm c} \Theta$, guaranteed by Lemma \ref{BoxCanonicalModels}.
    
    Suppose, towards contradiction, that $\Psi^{-}$ is inconsistent. Then, we have that $\vdash \bigwedge \Upsilon_0^{\Box_h} \to \neg \bigwedge \Theta_0^{\Box_v}$ for some finite subsets $\Upsilon_0\subseteq\Upsilon$ and $\Theta_0\subseteq\Theta$, and by Remark~\ref{RemarkK4xK4}\eqref{3RemarkK4xK4} we can deduce that
    \begin{equation}
    \label{EqConf2}
        \vdash\Diamond_v \bigwedge \Upsilon_0^{\Box_h} \to \neg \Box_v \bigwedge \Theta_0^{\Box_v}. 
    \end{equation}  
    On the other hand, by Remark~\ref{RemarkK4xK4}\eqref{1RemarkK4xK4}, we clearly have $\Box_h \bigwedge \Upsilon_0^{\Box_h} \in \Upsilon$. Since  $\Phi \prec_v^{\rm c} \Upsilon$, we deduce that $\Diamond_v \Box_h \bigwedge \Upsilon_0^{\Box_h} \in \Phi$. Hence, by the Church-Rosser axiom we show that $\Box_h \Diamond_v \bigwedge \Upsilon_0^{\Box_h} \in \Phi$. Now, from $\Phi \prec_h^{\rm c} \Theta$ using Lemma \ref{BoxCanonicalModels} we deduce $\Diamond_v \bigwedge \Upsilon_0^{\Box_h} \in \Theta$. Therefore, we get $\neg \Box_v \bigwedge \Theta_0^{\Box_v} \in \Theta$ by (\ref{EqConf2}), that is, $\Box_v \bigwedge \Theta_0^{\Box_v} \notin \Theta$. This concludes the proof by contradiction using Remark~\ref{RemarkK4xK4}\eqref{1RemarkK4xK4}.
    \end{description}

\item[\ref{itTrans}.] Suppose that $\Phi \prec_\alpha^{\rm c} \Psi \prec_\alpha^{\rm c} \Upsilon$ for $\alpha \in \{h, v\}$, and let $\varphi \in \Upsilon$. Then $\Diamond_\alpha \varphi \in \Psi$, and hence $\Diamond_\alpha \Diamond_\alpha \varphi \in \Phi$. 
By transitivity axiom and this fact that $\Phi$ is closed under derivability, we have $\Diamond_\alpha \varphi \in \Phi$. Thus, $\Phi \prec_\alpha^{\rm c} \Upsilon$. 

\item[\ref{TruthLemma}.] This is standard (see e.g.~\cite{Chagrov1997}) and can be proven by structural induction on $\varphi$, using the witness property~(Lemma~\ref{WitnessCanonicalModel}) for modal cases.\qedhere
\end{enumerate}
\end{proof}

The following result can be proven using Lemma~\ref{ExtoFCModel} and the completeness of $ [\mathsf{K4},\mathsf{K4}]^{\sf e}$ for its class of expanding domain models~\cite{mdml}, but it is illustrative to note that we can already derive it from our preliminary work.

\begin{thm} \label{Completeness$[K4,K4]^e$}
$ [\mathsf{K4},\mathsf{K4}]^{\sf e}$ is sound and complete for $ \sf{K4}\rtimes \sf{K4}$.
\end{thm}
\begin{proof}
 Soundness is given by Proposition \ref{SoundFor-Conflu}. To show completeness, suppose that $\not\vdash_{[{\sf K4},{\sf K4}]^{\sf e}} \varphi$ for some $\varphi \in \langfull$. Then, there is some maximal consistent set $\Phi$ in the canonical model such that $\varphi \notin \Phi$. By the truth lemma (Theorem \ref{Canonicalmodel$[K4,K4]^e$}\eqref{TruthLemma}), we deduce that $\mathcal{M}^{\rm c}, \Phi \not\Vdash \varphi$, showing that $\varphi$ is not valid for the class $\sf{K4}\rtimes\sf{K4}$.
\end{proof}

However, obtaining completeness for logics where one of the components is Noetherian will require a deeper analysis.

\section{Moments}\label{sec:Moments}

Since the canonical model is not an expanding domain model (or even an embedding domain model), we need to `extract' a model of this desired form from it. The decidability proof of Gabelaia et al.~\cite{pml} uses the well-established techniques of building {\em quasimodels}, possibly partial models where points are labelled by formulas intended to be true at that point. Following the terminology of more recent work (e.g.~\cite{FDMontacute}), quasimodels are made up of {\em moments} viewed as instances of vertical frames, suggestive of an intuition where the horizontal relation represents time. As our goal is to show completeness for expanding commutators with ${\sf GL}$ as the vertical component, our moments are specifically tailored for this logic. Each point in a moment is labelled by a type, as defined below.

\begin{defn}\label{defSType}
Let $\Sigma \subseteq \langfull$ be closed under subformulas and single negations; we write this as $\Sigma\Subset\langfull$. A {\em $\Sigma$-type} is any $\Phi \subseteq \Sigma$ satisfying the following properties:
\begin{enumerate}
    \item $\neg\varphi\in \Phi$ iff $\varphi\notin \Phi$, for every $\neg\varphi\in \Sigma$, and
    \item $\varphi \wedge \psi \in \Phi$ iff $\varphi, \psi \in \Phi$, for every $\varphi \wedge \psi \in \Sigma$.
\end{enumerate}
The set of all finite $\Sigma$-types is denoted by $\mathrm{T}_\Sigma$.
\end{defn}

We define for $\alpha \in  \{h, v\}$,
\begin{itemize}
    \item $\Phi \prec_\alpha^\Sigma \Psi$ iff whenever either $\varphi$ or $\Diamond_\alpha \varphi \in \Psi $ then $\Diamond_\alpha\varphi\in \Phi$, for every $\Diamond_\alpha\varphi\in\Sigma$. 
\end{itemize}

We adopt the convention that, for $\alpha \in \{h,v\}$, a map is a {\em homomorphism} with respect to $\prec_\alpha$ whenever it preserves the corresponding $\prec_\alpha$-relation of the structures under consideration.

\begin{defn}\label{DefQuasiMom}
Let $\Sigma\Subset\langfull$ (as in Definition~\ref{defSType}). A {\em $\Sigma$-quasimoment} is a structure $\mathfrak m=(|\mathfrak m|, \vr^\mathfrak m, \ell^\mathfrak m )$, where $(|\mathfrak m|,\vr^\mathfrak m)$ is a finite, tree-like strict partial order with root $r^\mathfrak m$ and $\ell^\mathfrak m\colon |\mathfrak m|\to T_\Sigma$ is a homomorphism with respect to $\vr$.

A {\em $\Sigma$-moment} is a $\Sigma$-quasimoment $\mathfrak m=(|\mathfrak m|, \vr^\mathfrak m, \ell^\mathfrak m )$ such that for every $w\in |\mathfrak m|$ and every $\vd\varphi\in\Sigma$, if  $\vd\varphi\in \ell^\mathfrak m(w)$, then there exists $v \succ_v^\mathfrak m w$ such that $\varphi\in \ell^\mathfrak m(v)$.

We denote the set of $\Sigma$-moments by $\mathbb M_\Sigma$.
\end{defn}

Moments are similar to vertical frames but with additional syntactic data. Accordingly, given two points that are horizontally related and a moment for the first one, we identify the conditions that a potential moment for the second point must satisfy.

\begin{defn}\label{ts}
Let $\Sigma\Subset\langfull$, and $\mathfrak{m}$ and $\mathfrak{n}$ be $\Sigma$-quasimoments. 
\begin{enumerate}
    \item Say that $\mathfrak n$ is a {\em successor} of $\mathfrak m$, denoted $\mathfrak m \hr^\Sigma \mathfrak n$, if there exists an embedding $e \colon |\mathfrak m| \to |\mathfrak n|$ such that for every $w \in |\mathfrak m|$, we have $\ell^\mathfrak m(w) \hr^{\rm \Sigma} \ell^\mathfrak n(e(w))$.
   \item Moreover, we write $\mathfrak m\sqsubseteq^\Sigma\mathfrak n$ if $\mathfrak m$ embeds into $\mathfrak n$ preserving both the root and the labeling, in which case we say that $ \mathfrak m$ is a {\em submoment} of $\mathfrak n$.
\end{enumerate}
\end{defn}

The key property of moments used in the decidability proof of Gabelaia et al.~\cite{pml} is that they are well-quasi-ordered, itself a consequence of Kruskal's theorem. This property will also be crucial for our completeness proof of $[\mathsf{GL},\mathsf{GL}]^{\sf e}$.

\begin{thm}[Kruskal \cite{Kruskal1960}]\label{thmKruskal}
For any finite set $\Lambda$, the class of finite $\Lambda$-labelled trees is well-quasi-ordered under embeddability, i.e., for every infinite sequence $(\mathfrak{T}_n)_{n<\omega}$, there exist $i<j<\omega$ such that $\mathfrak{T}_i$ embeds into $\mathfrak{T}_j$ via an embedding which preserves labels.
\end{thm}

As moments can be seen as $T_\Sigma$-labelled trees, we can immediately apply Kruskal's theorem to them to obtain the following.

\begin{corollary}\label{thmWQO}
Let $\Sigma\Subset\langfull$.
The set $\mathbb M_\Sigma$ is well-quasi-ordered by $\sqsubseteq^\Sigma$.
\end{corollary}

\section{Simulations}\label{sec:Simulations}

The extraction of quasimodels from the canonical model is performed via the use of simulations, which relate moments and maximal consistent sets. In this section, $\mathcal M^{\rm c}$ denotes the canonical model of a logic $[\mathsf {L}_h,\mathsf{GL}]^{\sf e}$, where, in principle, $\mathsf {L}_h$ is a normal extension of $\sf K4$. In practice, it may be assumed that $\mathsf {L}_h \in \{\mathsf{K4},\mathsf{GL}\}$.

\begin{defn}
Let $\Sigma\Subset\langfull$ and let $\mathfrak m$ be a $\Sigma$-quasimoment. A map $\sigma \colon |\mathfrak m| \to W^{\mathrm c}$ is called a {\em simulation} if it is a homomorphism with respect to $\vr $ and $\ell^{\mathfrak m}(w) = \sigma(w)\cap\Sigma$ for all $w \in |\mathfrak m|$.

For $\Phi \in W^{\mathrm c}$, we write $\mathfrak m \rightharpoonup  \Phi$ and say that $\mathfrak{m}$ {\em simulates} $\Phi$ if there exists a simulation $\sigma \colon |\mathfrak m| \to W^{\mathrm c}$ such that $\sigma(r^{\mathfrak m}) = \Phi$.
\end{defn}

Recall that, unlike quasimoments, moments are required to witness every formula of the form $\Diamond_v\varphi$ labelling a node by a suitable $\prec_v$-successor labelled with $\varphi$. For technical reasons, it is convenient to record instances where a diamond formula still requires such a witness, called `defects'. We also treat occurrences of $\Diamond_h$-formulas as horizontal defects, which are realized by extending the quasimoment with a successor.

\begin{defn}\label{defDefects}
Let $\Sigma\Subset\langfull$. A {\em potential defect} $\delta$ of a $\Sigma$-quasimoment $\mathfrak m$ is a pair $\delta = (w, \Diamond_\alpha \varphi)$ such that $\Diamond_\alpha \varphi \in \ell^{\mathfrak m}(w)$, where $\alpha \in \{h,v\}$. 
\begin{itemize}
    \item A potential defect $\delta$ is called a {\em horizontal defect} if $\alpha = h$. We say that a $\Sigma$-quasimoment $\mathfrak n$ {\em realizes} the horizontal defect $(w, \Diamond_h \varphi)$ if $\mathfrak m \hr^\Sigma\mathfrak n$ via an embedding $e : |\mathfrak{m}| \to |\mathfrak{n}|$ with $\varphi \in \ell^{\mathfrak n}(e(w))$.
    \item A potential defect $\delta$ is called a {\em vertical defect} if $\alpha = v$ and there is no $u \succ_v^{\mathfrak m} w$ such that $\varphi \in \ell^{\mathfrak m}(u)$.
\end{itemize}
\end{defn}

It will often be convenient to construct moments by first building quasimoments and then iteratively removing the vertical defects.
We may use this strategy to show that every maximal consistent set in the canonical model is simulated by some moment.

\begin{lem}\label{itMomWitNow}
Let $\Sigma\Subset\langfull$ and let $\Phi \in W^{\rm c}$. Then there exists a  $\Sigma$-moment $\mathfrak m$ such that $\mathfrak m \rightharpoonup \Phi$.
\end{lem}
\begin{proof}
To prove this, we construct a sequence $(\mathfrak{m}_i)_{i \in \mathbb{N}}$ of $\Sigma$-quasimoments and show that it eventually stabilizes at a $\Sigma$-moment $\mathfrak{m}$ such that $\mathfrak{m} \rightharpoonup \Phi$.

Consider the sequence $(\mathfrak{m}_i)_{i \in \mathbb{N}}$ defined recursively as:
\begin{itemize}
    \item $\mathfrak{m}_0 := (\{w\}, \prec_v^0, \ell^0)$, where $\prec_v^0 := \emptyset$ and $\ell^0(w) := \Phi \cap \Sigma$.
    \item Given $\mathfrak{m}_i = (|\mathfrak{m}_i|, \prec_v^i, \ell^i)$, the $\Sigma$-quasimoment $\mathfrak{m}_{i+1}$ is obtained by extending $\mathfrak{m}_i$ as follows: for every vertical defect $(u, \Diamond_v \varphi)$ of $\mathfrak{m}_i$ with $\ell^i(u) = \Psi \cap \Sigma$, introduce a fresh node $u'$ such that $u \prec_v^{i+1} u'$, and define $\ell^{i+1}(u') := \Psi' \cap \Sigma$ 
    where $\Psi' \in W^{\mathrm c}$ is chosen so that $\Psi \prec_v^{\mathrm c} \Psi'$ and $\varphi \wedge \neg \Diamond_v \varphi \in \Psi'$, as given by the witness property (Lemma~\ref{WitnessCanonicalModel}).
    
    In this case, we say that $u$ is a \emph{defect node of $\mathfrak{m}_i$}, and $u'$ \emph{realizes the defect} $(u, \Diamond_v \varphi)$ in $\mathfrak{m}_{i+1}$, or simply $u'$ \emph{realizes a defect at} $u$. 
\end{itemize}

We first see that the sequence is well-defined.

\begin{claim}
\label{Claim1}
$\mathfrak{m}_i$ is a $\Sigma$-quasimoment for any $i \in \mathbb{N}$.
\end{claim}

\begin{proofof}{Claim \ref{Claim1}}
We proceed by induction on $i \in \mathbb{N}$. The base case is immediate. For the inductive step, it suffices to verify that the homomorphism condition is preserved at the fresh nodes added in $\mathfrak{m}_{i+1}$, since it holds elsewhere by the inductive hypothesis.

Let $u' \in |\mathfrak{m}_{i+1}|$ be a fresh node realizing a defect at $(u, \Diamond_v \varphi)$ of $\mathfrak{m}_{i}$, with $\ell^{i+1}(u) = \Psi \cap \Sigma$ and $\ell^{i+1}(u') = \Psi' \cap \Sigma$, where $\Psi \prec_v^{\rm c} \Psi'$ and $\varphi \wedge \neg \Diamond_v \varphi \in \Psi'$. Let $\Diamond_v \psi \in \Sigma$. 
\begin{enumerate}
    \item If $\psi \in \Psi' \cap \Sigma$, then $\Diamond_v \psi \in \Psi$ by $\Psi \prec_v^{\rm c} \Psi'$. Hence, $\Diamond_v \psi \in \Psi \cap \Sigma$.
    \item If $\Diamond_v \psi \in \Psi' \cap \Sigma$, then $\Diamond_v \Diamond_v \psi \in \Psi$ by $\Psi \prec_v^{\rm c} \Psi'$. Since $\Psi$ is maximal consistent, transitivity yields $\Diamond_v \psi \in \Psi$. Thus, $\Diamond_v \psi \in \Psi \cap \Sigma$. \qedhere
\end{enumerate}
\end{proofof}

We next show that the sequence of $\Sigma$-quasimoments stabilizes, i.e. there exists $k \in \mathbb{N}$ such that $\mathfrak{m}_k = \mathfrak{m}_j$ for all $j \geq k$. The $\Sigma$-quasimoment $\mathfrak{m}_k$ is called the {\em limit} of $(\mathfrak{m}_i)_{i \in \mathbb{N}}$.

\begin{claim}
\label{Claim2} The sequence $(\mathfrak{m}_i)_{i \in \mathbb{N}}$ stabilizes.
\end{claim}

\begin{proofof}{Claim \ref{Claim2}}
Let $\mathrm{def}_i(u)$ denote the set of formulas $\Diamond_v \varphi$ such that $(u, \Diamond_v \varphi)$ is a vertical defect of $\mathfrak{m}_i$. It suffices to show that 
\begin{equation}
\label{eqCl2}
\mathrm{def}_{i+1}(u') \subsetneq \mathrm{def}_i(u)
\end{equation} 
for every $u' \in |\mathfrak{m}_{i+1}|$ realizing a defect node $u \in |\mathfrak{m}_i|$. Indeed, $\mathfrak{m}_i$ only extends properly to $\mathfrak{m}_{i+1}$ in order to realize existing defects. Since $\Sigma$ is finite, there are only finitely many defects in $\mathfrak{m}_0$. By \eqref{eqCl2}, whenever a fresh node is added to realize a defect node, its set of defects is strictly contained in that of the node it realizes. It follows that along any branch of the construction, the sets of defects form a strictly decreasing chain with respect to inclusion, so the construction must stabilize.

To show the inclusion in (\ref{eqCl2}), consider $\Diamond_v \psi \in \mathrm{def}_{i+1}(u')$ such that $u'$ is a fresh node realizing the defect $(u, \Diamond_v \varphi)$ in $\mathfrak{m}_i$. Let $\ell^{i+1}(u) = \Psi \cap \Sigma$ and $\ell^{i+1}(u') = \Psi' \cap \Sigma$, where $\Psi \prec_v^{\rm c} \Psi'$ and $\varphi \wedge \neg \Diamond_v \varphi \in \Psi'$. Since $\Diamond_v \psi \in \Psi'$ and $\Psi \prec_v^{\rm c} \Psi'$, it follows that $\Diamond_v \Diamond_v \psi \in \Psi$. As $\Psi$ is a maximal consistent set, transitivity yields $\Diamond_v \psi \in \Psi$, and hence $\Diamond_v \psi \in \Psi \cap \Sigma$. Moreover, it is straightforward to verify by induction on $i \in \mathbb{N}$ that defect nodes have no successors. Thus, $(u, \Diamond_v \psi)$ is, in particular, a vertical defect of $\mathfrak{m}_i$, showing that $\mathrm{def}_{i+1}(u') \subseteq \mathrm{def}_i(u)$. Finally, the inclusion is strict: indeed, $\Diamond_v \varphi \in \mathrm{def}_i(u)$ but by construction $\varphi \wedge \neg \Diamond_v \varphi \in \Psi'$. Hence, $\Diamond_v \varphi \notin \Psi'$, so we conclude that $\Diamond_v \varphi \notin \mathrm{def}_{i+1}(u')$.
\end{proofof}

Let $\mathfrak{m}$ denote the limit of $(\mathfrak{m}_i)_{i \in \mathbb{N}}$. We conclude that $\mathfrak{m}$ is the desired $\Sigma$-moment simulating $\Phi \in W^{\rm c}$.
\begin{claim}
\label{Claim3}
$\mathfrak{m}$ is a $\Sigma$-moment satisfying $\mathfrak{m}\rightharpoonup \Phi$.
\end{claim}

\begin{proofof}{Claim \ref{Claim3}}
To see that $\mathfrak{m}$ is a $\Sigma$-moment, it suffices to show that for every $u \in |\mathfrak{m}_i|$ with $\Diamond_v \varphi \in \ell^{i}(u)$, there exists $j \geq i$ and $u' \in \mathfrak{m}_j$ such that $u \prec_v^{j} u'$ and $\varphi \in \ell^j(u')$. The result follows by an easy induction on $i \in \mathbb{N}$, whose details are left to the reader.

Thus, we conclude $\mathfrak{m} \rightharpoonup \Phi$ via the simulation $\sigma : |\mathfrak{m}| \to W^{\mathrm c}$ defined by setting $\sigma(u) := \Psi$ for every $u \in |\mathfrak{m}|$ with $\ell^{\mathfrak{m}}(u) = \Psi \cap \Sigma$.
\end{proofof}
\qedhere
\end{proof}

Next, we show that every quasimoment admits an extension to a moment, preserving simulations, by iteratively applying the strategy of the previous lemma.

\begin{lem}\label{QuasiMomToMom}
Let $\Sigma\Subset\langfull$. If $\mathfrak m$ is a $\Sigma$-quasimoment and $\mathfrak m \rightharpoonup \Phi$ for some $\Phi \in W^{\rm c}$, then there exists a $\Sigma$-moment $ex(\mathfrak m)$ such that 
\[\mathfrak m \sqsubseteq^\Sigma ex(\mathfrak m) \quad\text{and}\quad ex(\mathfrak m) \rightharpoonup \Phi.\]
\end{lem}
\begin{proof}
Let $\mathfrak{m}$ be a $\Sigma$-quasimoment and let $\sigma: |\mathfrak{m}| \to W^{\rm c}$ be a simulation such that $\mathfrak{m} \rightharpoonup \Phi$ for some $\Phi \in W^{\rm c}$. We construct an extension $ex(\mathfrak{m}) \in \mathbb{M}_\Sigma$ satisfying the proper condition of the $\Sigma$-moments: for every node $w \in |ex(\mathfrak{m})|$ and $\vd \varphi \in\Sigma$,
    \begin{equation}\label{eqLemWomNow}
    \text{if } \vd \varphi \in \ell^{ex(\mathfrak{m})}(w), \text{ then there exists } u \succ_v^{ex(\mathfrak{m})} w \text{ with } \varphi \in \ell^{ex(\mathfrak{m})}(u).
    \end{equation}

    We define $ex(\mathfrak{m})$ recursively on the number of nodes of $\mathfrak{m}$ carrying vertical defects.
    \begin{itemize}
        \item If $\mathfrak{m}$ has no vertical defects, we simply set $ex(\mathfrak{m}) := \mathfrak{m}$.
        \item Otherwise, let $w \in |\mathfrak{m}|$ be a node with a vertical defect, and let $\ell^{\mathfrak{m}}(w) = \Psi \cap \Sigma$ for some $\Psi \in W^{\rm c}$. By Lemma~\ref{itMomWitNow}, there exist a $\Sigma$-moment $\mathfrak{m}_w$ and a simulation $\sigma_w : |\mathfrak{m}_w| \to W^{\rm c}$ such that $\mathfrak{m}_w \rightharpoonup \Psi$. We obtain a new $\Sigma$-quasimoment $\mathfrak{m}'$ from $\mathfrak{m}$ by grafting $\mathfrak{m}_w$ onto the node $w$, identifying $w$ with the root of $\mathfrak{m}_w$ while preserving the structure of $\mathfrak{m}$. We then set $ex(\mathfrak{m}) := ex(\mathfrak{m}')$.
\end{itemize}
A straightforward induction on the number of nodes of $\mathfrak{m}$ with vertical defects shows that $ex(\mathfrak{m})$ is a $\Sigma$-quasimoment extending $\mathfrak{m}$ and satisfying~\eqref{eqLemWomNow}, whose details are left to the reader. Hence, $ex(\mathfrak{m})$ is a $\Sigma$-moment.

Finally, the simulation $\sigma$ extends naturally along the construction. More precisely, letting $S \subseteq |\mathfrak{m}|$ be the set of nodes with vertical defects, we define
\[
\sigma_{ex} := \sigma \cup \bigcup_{w \in S} \sigma_w.
\]
Observe that each $\sigma_w$ agrees with $\sigma$ on their common domain (namely at $w \in S$), so the union is well defined. By construction, $\sigma_{ex} : |ex(\mathfrak{m})| \to W^{\rm c}$ is a simulation extending $\sigma$ and preserving its value at the root. In particular, $ex(\mathfrak{m}) \rightharpoonup \Phi$.
\end{proof}

The extraction of quasimodels is performed, roughly speaking, by taking limits of sequences of moments built by successively realizing horizontal defects while preserving simulations, as described in the following lemma.

\begin{lem}\label{itMomWitLater}
Let $\Sigma\Subset\langfull$ and let $\Phi\in W^{\rm c}$. If $\mathfrak m \rightharpoonup \Phi$ and $\mathfrak m$ has a horizontal defect $\delta$, then there exist $\Phi' \,\rh^{\mathrm c}\, \Phi$ and $\mathfrak m' \in \mathbb{M}_\Sigma$\ignore{$\mathfrak m'\rh^\Sigma \mathfrak m$ \sofia{I erased $\mathfrak m'\rh^\Sigma \mathfrak m$ here since this is part of the definition of realizing the horizontal defect.}} such that $\mathfrak m'$ realizes $\delta$ and $\mathfrak m'\rightharpoonup \Phi'$.   
\end{lem}
\begin{proof}
Let $\delta = (w, \Diamond_h \varphi)$ be a horizontal defect of $\mathfrak{m}$ and let $\sigma : |\mathfrak{m}| \to W^{\rm c}$ be the simulation such that $\mathfrak{m} \rightharpoonup \Phi$. We define recursively a sequence $\{(\mathfrak{m}_i, e_i, \sigma_i)\}_{i \in \mathbb{N}}$ such that, for each $i \in \mathbb{N}$, $\mathfrak{m}_i$ is a $\Sigma$-quasimoment, $e_i$ is a partial embedding, and $\sigma_i : |\mathfrak{m}_i| \to W^{\rm c}$ is a simulation satisfying
\[
\sigma(u) \prec_h^{\rm c} \sigma_i(e_i(u)) \quad \text{for all } u \in \mathrm{Dom}(e_i)
\]
where $\mathrm{Dom}(e_i)$ denotes the domain of the partial embedding $e_i$.
\begin{enumerate}[label=$\bullet$, wide, labelwidth=!, labelindent=5pt, itemsep=0pt]
    \item We first define $\mathfrak{m}_0 := (\{w'\}, \prec_v^0, \ell^0)$, where $\prec_v^0 := \emptyset$ and $\ell^0(w') := \Psi \cap \Sigma$ for some $\Psi \in W^{\rm c}$ such that $\Psi \succ_h^{\rm c} \sigma(w)$, obtained by the witness Lemma~\ref{WitnessCanonicalModel} from the fact that $\Diamond_h \varphi \in \ell^{\mathfrak{m}}(w) = \sigma(w) \cap \Sigma$. Define $e_0(w) := w'$ and $\sigma_0(w') := \Psi$.
    \item Given $(\mathfrak{m}_i, e_i, \sigma_i)$, we construct $(\mathfrak{m}_{i+1}, e_{i+1}, \sigma_{i+1})$ by extending $\mathfrak{m}_i$, $e_i$, and $\sigma_i$ as follows:

    For each $v \in |\mathfrak{m}| \setminus \mathrm{Dom}(e_i)$ such that either $v$ is an immediate $\prec_v^{\mathfrak m}$-predecessor or an immediate $\prec_v^{\mathfrak m}$-successor of some $u\in\mathrm{Dom}(e_i)$, that is,
    \[
    v \prec_v^{\mathfrak m} u
    \quad\text{and there is no }w \neq u, v \text{ with }
    v \prec_v^{\mathfrak m} w \prec_v^{\mathfrak m} u
    \]
    or
    \[
    u \prec_v^{\mathfrak m} v
    \quad\text{and there is no }w \neq u, v\text{ with }
    u \prec_v^{\mathfrak m} w \prec_v^{\mathfrak m} v,
    \]
we add a fresh node $v'$ and extend $e_i$ by setting $e_{i+1}(v):=v'$. The position of $v'$ in $\mathfrak{m}_{i+1}$ is determined by:
    \begin{enumerate}
    \item[(i)] $v' \prec_v^{{i+1}} e_i(u)$ if $v \prec_v^{\mathfrak{m}} u$;
    \item[(ii)] $v' \succ_v^{{i+1}} e_i(u)$ if $v \succ_v^{\mathfrak{m}} u$.
    \end{enumerate}
    
    Using the properties (LC) and (CR) of the canonical model, we define the label and simulation for $v'$ by setting $\ell^{{i+1}}(v') := \Upsilon \cap \Sigma$ and $\sigma_{i+1}(v') := \Upsilon$, where $\Upsilon \in W^{\rm c}$ satisfies one of the following conditions:
    \begin{enumerate}
    \item[(i)] $\sigma(v) \prec_h^{\rm c} \Upsilon \prec_v^{\rm c} \sigma_i(e_i(u))$ by the property (LC) if $v \prec_v^{\mathfrak{m}} u$ given that $\sigma(v) \prec_v^{\rm c} \sigma(u) \prec_h^{\rm c} \sigma_i(e_i(u))$;
    \item[(ii)] $\sigma(v) \prec_h^{\rm c} \Upsilon \succ_v^{\rm c} \sigma_i(e_i(u))$ by the property (CR) if $v \succ_v^{\mathfrak{m}} u$ given that $\sigma(v) \succ_v^{\rm c} \sigma(u) \prec_h^{\rm c} \sigma_i(e_i(u))$.
    \end{enumerate}
\end{enumerate}

\begin{figure}[t]
\centering
\small
\begin{tikzpicture}
	\begin{pgfonlayer}{nodelayer}
    	\node [style=none, label=below:$\mathfrak{m}$] (0) at (-1.5, 0.5) {};
		\node [style=none] (1) at (-3, 4) {};
		\node [style=none] (2) at (0, 4) {};
    	\node [style=none] (3) at (-2, 3) {};
		\node [style=none] (5) at (-0.75, 2.75) {};
		\node [style=none, label=right: {\scriptsize $\mathrm{Dom}(e_i)$}] (8) at (-1.5, 1.75) {};
		\node [style=sroot, label=left:$u$] (9) at (-1.5, 2.5) {};
		\node [style=sroot, label=above:$v$, red] (10) at (-1.5, 3.5) {};
		\node [style=sroot, label=below:$v$, red] (11) at (-1.5, 1.25) {};
		\node [style=sroot, label=above:$v'$, red] (19) at (1.5, 3.5) {};
		\node [style=sroot, label=below:$v'$, red] (20) at (1.5, 1.25) {};
		\node [style=none] (21) at (1, 3) {};
		\node [style=none] (22) at (2.25, 2.75) {};
		\node [style=none, label=right:$\mathfrak{m}_i$] (23) at (1.5, 1.75) {};
		\node [style=sroot, label=right:$e_i(u)$] (24) at (1.5, 2.5) {};

        \draw [style=edgeto] (9) to (10);
		\draw [style=edgeto] (11) to (9);
		\draw [style=edgeto, red, dashed] (10) to (19);
		\draw [style=edgeto, red, dashed] (11) to (20);

		\draw[style=edgeto] (9) to (24);
		\draw[style=edgeto, red, dashed] (24) to (19);
		\draw[style=edgeto, red, dashed] (20) to (24);
	\end{pgfonlayer}
\begin{pgfonlayer}{background}

\fill[black!10, opacity=0.7]
    (0.center)
        to (1.center)
        to[in=-150, out=30, looseness=0.75] (2.center)
        to (0.center)
    -- cycle;

		\draw [black!20] (0.center) to (1.center);
		\draw [in=-150, out=30, looseness=0.75, black!20] (1.center) to (2.center);
		\draw [black!20] (2.center) to (0.center);
    
\end{pgfonlayer}

\begin{pgfonlayer}{edgelayer}
\fill[red!10, opacity=0.7]
    (5.center)
        to[in=60, out=-165] (8.center)
        to[in=-90, out=120, looseness=1.25] (3.center)
        to[in=120, out=-15, looseness=1.25,] (5.center)
    -- cycle;

\fill[red!10, opacity=0.7]
    (22.center)
        to[in=60, out=-165] (23.center)
        to[in=-90, out=120, looseness=1.25] (21.center)
        to[in=120, out=-15, looseness=1.25,] (22.center)
    -- cycle;

    \draw [in=60, out=-165, red!20] (5.center) to (8.center);
	\draw [in=-90, out=120, looseness=1.25, red!20] (8.center) to (3.center);
	\draw [in=120, out=-15, looseness=1.25, red!20] (3.center) to (5.center);

    \draw [in=60, out=-165, red!20] (22.center) to (23.center);
	\draw [in=-90, out=120, looseness=1.25, red!20] (23.center) to (21.center);
	\draw [in=120, out=-15, looseness=1.25, red!20] (21.center) to (22.center);
\end{pgfonlayer}
\end{tikzpicture}
\caption{Inductive step in the construction of the sequence $\{(\mathfrak{m}_i, e_i, \sigma_i)\}_{i \in \mathbb{N}}$.}
\end{figure}
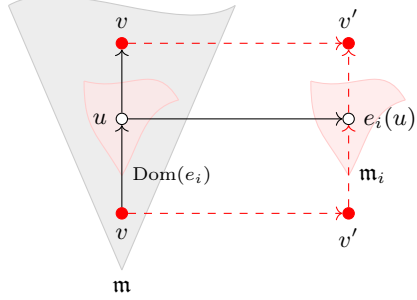

First, we show that the construction is well-defined.

\begin{claim}\label{Claim1itMomWitLater}
The sequence $\{(\mathfrak{m}_i, e_i, \sigma_i)\}_{i \in \mathbb{N}}$ is well-defined.
\end{claim}

\begin{proofof}{Claim \ref{Claim1itMomWitLater}}
More precisely, we show for every $i \in \mathbb{N}$ that
\begin{enumerate}[label=(\alph*)]
    \item\label{C4a}  $\mathrm{Dom}(e_i) \subseteq |\mathfrak{m}|$ equipped with the relation $\prec_v^{\mathfrak{m}}$ restricted to it forms a tree;
    \item\label{C4b}  $e_i : \mathrm{Dom}(e_i) \to |\mathfrak{m}_i|$ is a $\prec_v$-isomorphism such that $\mathrm{Im}(e_i) = |\mathfrak{m}_i|$, where $\mathrm{Im}(e_i)$ denotes the image of $e_i$;
    \item\label{C4c}  $\mathfrak{m}_i$ is a $\Sigma$-quasimoment;
    \item\label{C4d}  $\sigma_i : |\mathfrak{m}_i| \to W^{\rm c}$ is a simulation satisfying $\sigma(u) \prec_h^{\rm c} \sigma_i(e_i(u))$ for every $u \in \mathrm{Dom}(e_i)$.
\end{enumerate}
The proof proceeds by induction on $i \in \mathbb{N}$. The base case is immediate from the construction. For the inductive step, assume that $(\mathfrak{m}_i, e_i, \sigma_i)$ satisfies \ref{C4a}--\ref{C4d}, and consider the extension to $(\mathfrak{m}_{i+1}, e_{i+1}, \sigma_{i+1})$ as defined in the sequence. We provide details for the first three properties, while the verification of \ref{C4d} follows directly from the definition of the recursive construction and is left to the reader.

\begin{enumerate}[wide, labelwidth=!, labelindent=5pt, itemsep=0pt]
    \item[\ref{C4a}] By the inductive hypothesis, $\mathrm{Dom}(e_i)$ is a tree. The map $e_{i+1}$ is defined to extend $e_i$ with nodes $v \in |\mathfrak{m}| \setminus \mathrm{Dom}(e_i)$ which are immediate $\prec_v^\mathfrak{m}$-predecessors or $\prec_v^\mathfrak{m}$-successors of some $u \in \mathrm{Dom}(e_i)$, together with the edge witnessing this relation. If $v \succ_v^\mathfrak{m} u$, it is straightforward that the extension is still a tree. Otherwise, if $v \prec_v^\mathfrak{m} u$, since $v \notin \mathrm{Dom}(e_i)$ but $u \in \mathrm{Dom}(e_i)$, we deduce that $u$ is the root of the tree $\mathrm{Dom}(e_i)$. Therefore, we have the extension is also a tree for this case. 

    \item[\ref{C4b}] By construction, every new node $v'$ of $\mathfrak{m}_{i+1}$ arises from a unique $v \in |\mathfrak{m}| \setminus \mathrm{Dom}(e_i)$ which is immediate adjacent (in the $\prec_v$-tree structure) to some $u \in \mathrm{Dom}(e_i)$. This element $v$ is added to the domain of $e_{i+1}$ and mapped to $v'$ via $e_{i+1}(v) := v'$, so that $\mathrm{Im}(e_{i+1}) = |\mathfrak{m}_{i+1}|$. Since $e_i$ is a $\prec_v$-isomorphism by the inductive hypothesis and the construction introduces fresh nodes exactly corresponding to previously unmapped elements while respecting adjacency in the tree structure, it follows that $e_{i+1}$ is again bijective. Moreover, the extension is defined so as to preserve the $\prec_v$-relation, and therefore $e_{i+1}$ is a $\prec_v$-isomorphism.

    \item[\ref{C4c}] By \ref{C4a}, $\mathrm{Dom}(e_{i+1})$ is a tree, and by \ref{C4b} so is $(|\mathfrak{m}_{i+1}|, \prec_v^{i+1})$. It remains to verify that the labelling satisfies the homomorphism condition: this follows from the inductive hypothesis together with the definition of the labels of the newly added nodes and the transitivity properties of maximal consistent sets. The details are routine and left to the reader. \qedhere
\end{enumerate}
\end{proofof}

Next, we show that the sequence stabilizes, meaning that it is eventually constant at some element, which we call its {\em limit}.

\begin{claim}\label{Claim3itMomWitLater}
The sequence $\{(\mathfrak{m}_i, e_i, \sigma_i)\}_{i \in \mathbb{N}}$ stabilizes.
\end{claim}

\begin{proofof}{Claim \ref{Claim3itMomWitLater}}
Since the construction properly extends by adding fresh nodes corresponding to nodes in $\mathfrak{m}$ that are not in the domain of the corresponding embedding, it suffices to show that there exists $k \in \mathbb{N}$ such that $\mathrm{Dom}(e_k) = |\mathfrak{m}|$.

To this end, observe that $\mathrm{Dom}(e_i) \subseteq \mathrm{Dom}(e_{i+1})$ for all $i \in \mathbb{N}$ by a straightforward induction on $i \in \mathbb{N}$. Moreover, whenever $\mathrm{Dom}(e_i) \subsetneq |\mathfrak{m}|$, the construction adds at least one new element, so that 
$\mathrm{Dom}(e_i) \subsetneq \mathrm{Dom}(e_{i+1})$. Therefore, since $|\mathfrak{m}|$ is finite, there exists $k \in \mathbb{N}$ such that $\mathrm{Dom}(e_k) = |\mathfrak{m}|$, and the sequence stabilizes. \qedhere
\end{proofof}

Let $(\hat{\mathfrak{m}}, \hat{e}, \hat{\sigma})$ denote the limit of the sequence $\{(\mathfrak{m}_i, e_i, \sigma_i)\}_{i \in \mathbb{N}}$. We show that this limit is already a successor realizing the horizontal defect $\delta$, and that it simulates a maximal consistent set that is a horizontal successor of the one simulated by $\mathfrak{m}$.

\begin{claim}\label{Claim2itMomWitLater}
\begin{enumerate} 
    \item\label{C2succR} $\hat{\mathfrak{m}}$ is a successor of $\mathfrak{m}$ realizing $\delta$.
    \item\label{C2ssim} $\hat{\mathfrak{m}} \rightharpoonup \Phi'$ for some $\Phi' \succ_h^{\rm c} \Phi$.
\end{enumerate}
\end{claim}

\begin{proofof}{Claim \ref{Claim2itMomWitLater}}
\begin{enumerate}[wide, labelwidth=!, labelindent=5pt, itemsep=0pt] 
    \item[\ref{C2succR}.] It suffices to show that $\mathfrak{m} ,\hr^\Sigma, \hat{\mathfrak{m}}$ via the embedding $\hat{e}$. Once this is established, we obtain that $\varphi \in \ell^{\hat{\mathfrak{m}}}(\hat{e}(w))$, since $\varphi \in \ell^{i}(e_i(w))$ for every $i \in \mathbb{N}$. The latter fact follows by a straightforward induction on $i$, whose details are left to the reader.

    As observed in Claim \ref{Claim3itMomWitLater}, we have $\mathrm{Dom}(\hat{e}) = |\mathfrak{m}|$, so $\hat{e} : |\mathfrak{m}| \to |\hat{\mathfrak{m}}|$ is a (total) embedding. To verify that $\mathfrak{m} \,\hr^\Sigma\, \hat{\mathfrak{m}}$, let $\Diamond_h \psi \in \Sigma$ and suppose that either $\psi \in \ell^{\hat{\mathfrak{m}}}(\hat{e}(u))$ or $\Diamond_h \psi \in \ell^{\hat{\mathfrak{m}}}(\hat{e}(u))$ for some $u \in |\mathfrak{m}|$. Since $\ell^{\hat{\mathfrak{m}}}(\hat{e}(u)) = \hat{\sigma}(\hat{e}(u)) \cap \Sigma$ and $\sigma(u) \prec_h^{\rm c} \hat{\sigma}(\hat{e}(u))$, it follows from the definition of $\prec_h^{\rm c}$ that either $\Diamond_h \psi \in \sigma(u)$ or $\Diamond_h \Diamond_h \psi \in \sigma(u)$. In the latter case, by transitivity and maximal consistency we also obtain $\Diamond_h \psi \in \sigma(u)$. Hence, in both cases we have that $\Diamond_h \psi \in \sigma(u)$, and therefore we conclude $\Diamond_h \psi \in \ell^{\mathfrak{m}}(u)$.

    \item[\ref{C2ssim}.] Let $\Phi' := \hat{\sigma}(r^{\hat{\mathfrak{m}}})$. Since $r^{\hat{\mathfrak{m}}} = \hat{e}(r^{\mathfrak{m}})$, and recalling that $\sigma(r^{\mathfrak{m}}) = \Phi$, it follows from the construction that $\Phi' \succ_h^{\rm c} \Phi$. Hence $\hat{\mathfrak{m}} \rightharpoonup \Phi'$ as required. \qedhere
\end{enumerate}
\end{proofof}

Finally, we extend the $\Sigma$-quasimoment $\hat{\mathfrak{m}}$ to a $\Sigma$-moment $\mathfrak{m}'$ such that $\mathfrak{m}' \rightharpoonup \Phi'$ using Lemma~\ref{QuasiMomToMom}. Moreover, $\hat{e} : |\mathfrak{m}| \to |\mathfrak{m}'|$  witnesses that $\mathfrak m\hr^\Sigma \mathfrak m'$, showing that $\mathfrak{m}'$ realizes the horizontal defect $\delta$. This concludes the proof.
\end{proof}

\section{Quasimodel extraction}\label{sec:Extract}

Next, we show how simulations can be used to extract quasimodels from the canonical model, leading to a proof of completeness for $[\mathsf{K4},\mathsf{GL}]^{\mathsf e}$. A quasimodel is an intermediate structure that resembles an ${\sf e}$-frame, in which the horizontal frame forms a tree and each point carries an associated moment, ensuring realizability of the horizontal defects.

\begin{defn} \label{def:Quasimodel}
Let $\Sigma\Subset\langfull$. A {\em weak $\Sigma$-quasimodel} is a tuple $\mathfrak Q = (T^\mathfrak{Q} ,\hr^\mathfrak{Q} , \mu^\mathfrak{Q})$, where 
\begin{enumerate}
    \item $(T^\mathfrak{Q} ,\hr^\mathfrak{Q} ) $ is a transitive, irreflexive tree with root $r^\mathfrak{Q}$ and
    \item $\mu^\mathfrak{Q} \colon T^\mathfrak{Q} \to \mathbb M_\Sigma$ is a homomorphism with respect to $\prec_h$.
\end{enumerate}
We say that $\mathfrak Q$ is a {\em $\Sigma$-quasimodel} if for every $t \in T^\mathfrak{Q}$ and every horizontal defect $\delta$ of the $\Sigma$-moment $\mu^\mathfrak{Q}(t)$, there exists $s\rh^\mathfrak{Q} t$ such that $\mu^\mathfrak{Q}(s)$ realizes $\delta$.
\end{defn}

We say that a formula $\varphi$ is {\em satisfied} in a $\Sigma$-quasimodel $\mathfrak Q$ if there exist $t \in T^\mathfrak{Q}$ and $w \in |\mu^\mathfrak{Q}(t)|$ such that $\varphi \in \ell^{\mu^\mathfrak{Q}(t)}(w)$. Other related notions are defined on $\Sigma$-quasimodels in the usual way.

Recall that moments as we have defined them are specifically tailored for $\sf GL$ and, since the horizontal relation on a quasimodel is assumed transitive and irreflexive, it will always be a $\sf K4$ relation, and a $\sf GL$ relation when the quasimodel is finite. With this in mind, the following should not be surprising.

\begin{lem}\label{lemQuasiToMod}
Let $\Sigma\Subset\langfull$. If a formula $\varphi \in \Sigma$ is satisfied in a $\Sigma$-quasimodel $\mathfrak Q$, then it is satisfied in a $(\mathsf{K4}\rtimes \mathsf{GL})^{\sf e}$-model $\mathfrak M$.
If moreover $\mathfrak Q$ is finite, then $\mathfrak M$ can be taken to be a $(\mathsf{GL}\rtimes \mathsf{GL})^{\sf e}$-model.
\end{lem}
\begin{proof}
 Let $\mathfrak Q = (T^\mathfrak{Q} ,\hr^\mathfrak{Q} , \mu^\mathfrak{Q})$ be a $\Sigma$-quasimodel. By definition, whenever $t, s \in T^\mathfrak{Q}$ with $t \hr^\mathfrak{Q} s$, we have $\mu^\mathfrak{Q}(t) \prec_h^{\Sigma} \mu^\mathfrak{Q}(s)$, meaning that $\mu^\mathfrak{Q}(t)$ embeds into $\mu^\mathfrak{Q}(s)$ via some embedding $e_{ts} : |\mu^\mathfrak{Q}(t)| \to |\mu^\mathfrak{Q}(s)|$. In particular, for all $u, w \in |\mu^\mathfrak{Q}(t)|$, we have
\[
u \prec_v^{\mu^\mathfrak{Q}(t)} w \quad \text{iff} \quad e_{ts}(u) \prec_v^{\mu^\mathfrak{Q}(s)} e_{ts}(w).
\]
Thus, the triple $(\mathfrak F, \mu^\mathfrak{Q}, e)$ where $\mathfrak F := (T^\mathfrak{Q}, \hr^\mathfrak{Q})$ and $e$ assigns to each pair $t, s \in T^\mathfrak{Q}$ with $t \hr^\mathfrak{Q} s$ an embedding $e_{ts}$, forms a $(\mathsf{K4}\rtimes \mathsf{GL})^{\sf e}$-frame. Moreover, the frame is finite if $\mathfrak Q$ is, hence in this case, it is a $(\mathsf{GL}\rtimes \mathsf{GL})^{\sf e}$-frame.

We can then define a model
$\mathfrak M := (\mathfrak F, \mu^\mathfrak{Q}, e, V)$, where $V=(V^t)_{t\in T^{\mathfrak Q}}$ is the family of valuations determined by
\[
w \in V^t(p) \quad \text{iff} \quad p \in \ell^{\mu^\mathfrak{Q}(t)}(w)
\quad \text{for } p \in \mathrm{Prop}
\] 
for each $t \in T^\mathfrak{Q}$ and $w \in |\mu^\mathfrak{Q}(t)|$.
A straightforward induction on $\varphi \in \langfull$ shows that $\varphi \in \ell^{\mu^\mathfrak{Q}(t)}(w)$ iff $\mathfrak{M}, (t, w) \Vdash \varphi$ for any $t \in T^\mathfrak{Q}$ and $w \in |\mu^\mathfrak{Q}(t)|$, concluding the proof.\qedhere
\end{proof}

By the preceding result, to show completeness of the expandind commutators, it suffices to see that consistent formulas are satisfied in quasimodels. To bridge the gap between consistency and satisfiability, we equip quasimodels with morphisms that, via simulations, associate moments with maximal consistent sets.

\begin{defn}
\label{def:Qmorphism}
Let $\Sigma\Subset\langfull$, let $\mathsf L_h$, $\mathsf L_v$ be unimodal logics, let $\mathcal{M}^{\rm c}$ be the canonical model of the expanding commutator $[{\sf L}_h, {\sf L}_v]^{\sf e}$ and let $\mathfrak{Q}$ be a weak $\Sigma$-quasimodel. A $\mathfrak{Q}$-morphism for $[{\sf L}_h, {\sf L}_v]^{\sf e}$ is a function $\pi^\mathfrak{Q} \colon T^\mathfrak{Q}\to W^{\rm c}$ which is a $\prec_h$-homomorphism and satisfies $\mu^\mathfrak{Q}(t) \rightharpoonup \pi^\mathfrak{Q}(t)$ for any $t\in T^\mathfrak{Q}$.
\end{defn}

The core of our quasimodel extraction is given by the following lemma.

\begin{lem}\label{lemExistsQuasiK4}
Let $\Sigma\Subset\langfull$ and let $\mathcal M^{\rm c}$ be the canonical model of an expanding commutator of the form $[\mathsf L_h,\mathsf{GL}]^{\sf e}$, where $\mathsf{K4}\subseteq \mathsf L_h$.
If $\Phi \in W^{\rm c}$, then there exists a $\Sigma$-quasimodel $\mathfrak{Q}$ and a $\mathfrak{Q}$-morphism $\pi^\mathfrak{Q}$ for $[{\sf L}_h, {\sf GL}]^{\sf e}$ such that $\pi^\mathfrak{Q}(r^\mathfrak{Q}) = \Phi$.
\end{lem}
\begin{proof} 
We define a sequence $\{(\mathfrak{Q_i}, \pi^{i})\}_{i\in \mathbb{N}}$, where each $\mathfrak{Q}_i$ is a weak $\Sigma$-quasimodel and $\pi^{i}$ is a $\mathfrak{Q}_i$-morphism for $[{\sf L}_h, {\sf GL}]^{\sf e}$.

First, let $\mathfrak Q_0 := (\{r^0\}, \emptyset, \mu^0)$ consisting of a single point such that $\pi^0(r^0) := \Phi$, and choose $\mu^0(r^0)$ so that it satisfies $\mu^0(r^0) \rightharpoonup \Phi$ according to Lemma~\ref{itMomWitNow}.

Now assume that $\mathfrak Q_i$ and $\pi^{i}$ have been defined. For every horizontal defect $\delta$ of some $\mu^{i}(t)$ with $t \in T^{i}$ that is not realized by any moment of $\mathfrak{Q}_i$, proceed as follows. Add a fresh successor $s \succ^{i+1}_h t$. Choose $\Phi' \succ_h^{\rm c} \pi^i(t)$ and a $\Sigma$-moment $\mathfrak{m}'$ as provided by Lemma~\ref{itMomWitLater}. Set $\pi^{i+1}(s):=\Phi'$ and $\mu^{i+1}(s) := \mathfrak{m}'$.

Let $(\mathfrak Q_\infty, \pi^\infty)$ be the direct limit of $\{(\mathfrak Q_i, \pi^{i})\}_{i \in \mathbb{N}}$.
By construction, $\mathfrak Q_\infty$ is a weak $\Sigma$-quasimodel satisfying $\pi^{\infty}(r^{\infty}) = \Phi$. Moreover, it is a $\Sigma$-quasimodel since every horizontal defect of its moments is realized. \qedhere
\end{proof}

\begin{thm} \label{CompletenessK4GL}
The logic $[\sf{K4},\sf{GL}]^{\sf e}$ is sound and complete with respect to the following classes:
\begin{enumerate}
\item the class of expanding domain frames $(\sf{K4}\times \sf{GL})^{\sf e}$;
\item the class of embedding domain frames $(\sf{K4}\rtimes \sf{GL})^{\sf e}$;
\item the class of forward-confluent frames $\sf{K4}\rtimes \sf{GL}$.
\end{enumerate}
\end{thm}
\begin{proof}
The soundness with respect to these classes follows from the soundness of the logics $\sf K4$ and $\mathsf{GL}$, together with Proposition~\ref{SoundFor-Conflu} and Corollary~\ref{SoundFor-expan&emdedd}.

For completeness, let $\varphi$ be any consistent formula of $\langfull$. Then $\varphi \in \Phi$ for some $\Phi \in W^{\mathrm c}$, by the Lindenbaum's lemma. Let $\Sigma$ be the set of subformulas of $\varphi$. By Lemma~\ref{lemExistsQuasiK4}, there exists a $\Sigma$-quasimodel $\mathfrak{Q}$ and a $\mathfrak{Q}$-morphism $\pi^\mathfrak{Q}$ such that $\pi^\mathfrak{Q}(r^\mathfrak{Q}) = \Phi$. Let $\mathfrak{m} := \mu^\mathfrak{Q}(r^\mathfrak{Q})$. Since $\mathfrak{m} \rightharpoonup \Phi$, it follows from the definition of a simulation  that $\ell^\mathfrak{m}(r^\mathfrak{m}) = \Phi \cap \Sigma$. As $\varphi \in \Sigma$ and $\varphi \in \Phi$, we obtain $\varphi \in \ell^\mathfrak{m}(r^\mathfrak{m})$, that is, $\varphi$ is satisfiable in $\mathfrak Q$. Thus, by Lemma~\ref{lemQuasiToMod}, $\varphi$ is satisfied in an embedding domain model. The completeness with respect to the expanding domain frames and the forward-confluent frames is obtained from Lemma~\ref{EquiEmbe&Expan} and Lemma~\ref{ExtoFCModel}, respectively.
\end{proof}

\section{Yankov-Fine formulas}\label{Sec:Char}

Next we turn our attention to $[\mathsf{GL},\mathsf{GL}]^{\sf e}$. The work we have done so far already represents a considerable portion of our completeness proof, but a crucial extra step is required that was not needed for $[\mathsf{K4},\mathsf{GL}]^{\sf e}$.
The issue here is that $\prec^{\rm c}_h$ is not conversely well-founded, and hence the quasimodel extraction method may in principle yield quasimodels with infinite branches. However, {\em definable} subsets of the canonical model do have terminal elements due to the Löb axiom, and we can exploit this property to avoid infinite branches. The specific properties we need to define are given by a variant of Yankov-Fine formulas \cite{yankov1963relation, Fine1974}.

Let $\Sigma\Subset\langfull$. For a $\Sigma$-moment $\mathfrak m$ and $w\in |\mathfrak m|$, the {\em height} of $w$ is the maximal $N \in \mathbb{N}$ such that there exists a sequence $w = w_0  \vr^\mathfrak m w_1  \vr^\mathfrak m\dots  \vr^\mathfrak m w_N$ in $|\mathfrak m|$. The height of $\mathfrak m$ is defined as the height of its root $r^{\mathfrak m}$. The formula ${\rm Sim}(\mathfrak{m})$ is defined by induction on the height of $\mathfrak m$ as follows:
  \[
  {\rm Sim}(\mathfrak{m}) := \bigwedge \ell^\mathfrak{m}(r^\mathfrak{m}) \wedge \bigwedge_{r^\mathfrak{m} \vr^\mathfrak m r'} \Diamond_v {\rm Sim}(\mathfrak{m}\!\upharpoonright_{r'})
  \]
where $\mathfrak{m}\!\upharpoonright_{r'}$ denote the subtree of $\mathfrak{m}$ rooted at $r' \succ_v^\mathfrak{m} r^\mathfrak{m}$, i.e., the restriction of $\mathfrak{m}$ to the set $\{v \mid r' \peq_v^\mathfrak m v\}$ for $\peq_v^\mathfrak m$ the reflexive closure of $\prec_v^\mathfrak{m}$. The following fact is a special case of Theorem~7.1 in~\cite{Fernandez11}.

\begin{thm}\label{thmYF}
Let $\Sigma\Subset\langfull$, let $\mathsf L_h$, $\mathsf L_v$ be unimodal logics such that ${\sf L}_v \supseteq {\sf GL}$ and let $\mathcal M^{\rm c}$ be the canonical model for $[ \mathsf L_h, \mathsf L_v]^{\sf e}$. For any $\Sigma$-moment $\mathfrak m$ and any $\Phi \in W^{\rm c}$, the following are equivalent.
\begin{enumerate}
    \item\label{YFsim} $\mathfrak{m}\rightharpoonup \Phi$.
    \item\label{YFSimSat} $\Phi \in \lb{\rm Sim}(\mathfrak m)\rb^{\rm c}$.
    \item\label{YFSimTruthLemma} ${\rm Sim}(\mathfrak{m}) \in \Phi$. 
\end{enumerate}
\end{thm}
\begin{proof}
By the truth lemma (Theorem \ref{Canonicalmodel$[K4,K4]^e$}\eqref{TruthLemma}), the statements in (\ref{YFSimSat}) and (\ref{YFSimTruthLemma}) are equivalent. Hence, it suffices to see the $\mathfrak{m}\rightharpoonup \Phi $ iff ${\rm Sim}(\mathfrak{m}) \in \Phi$ by induction on the height of $\mathfrak{m}$. 
\begin{enumerate}[wide, labelwidth=!, labelindent=5pt, itemsep=0pt] 
\item[$(\Rightarrow)$]
The base case is immediate since $\mathfrak{m}$ is simply the node $r^\mathfrak{m}$ and we clearly have $\bigwedge \ell^\mathfrak{m}(r^\mathfrak{m}) \in \Phi$ by definition. For the inductive step, it suffices to show that $\Diamond_v \mathrm{Sim}(\mathfrak{m}\!\upharpoonright_{r'}) \in \Phi$ for every $r' \succ_v^{\mathfrak{m}} r^{\mathfrak{m}}$, so using the argument for the base case we conclude that $\mathrm{Sim}(\mathfrak{m}) \in \Phi$. Hence, let $\sigma : |\mathfrak{m}| \to W^{\rm c}$ be the simulation for which $\sigma(r^\mathfrak{m}) = \Phi$, and consider $r' \succ_v^{\mathfrak{m}} r^{\mathfrak{m}}$. By the definition of simulation, we have that $\sigma(r') \succ_v^{\rm c} \Phi$. Moreover, the restriction of $\sigma$ to $\mathfrak{m}\!\upharpoonright_{r'}$ is again a simulation so, by the inductive hypothesis, we obtain ${\rm Sim}(\mathfrak{m}\!\upharpoonright_{r'}) \in \sigma(r')$. Thus, from $\sigma(r') \succ_v^{\rm c} \Phi$ we conclude $\Diamond_v\, {\rm Sim}(\mathfrak{m}\!\upharpoonright_{r'}) \in \Phi$.

 \item[$(\Leftarrow)$] The base case is satisfied by defining the simulation $\sigma : |\mathfrak{m}| \to W^{\rm c}$ as $\sigma(r^\mathfrak{m}) := \Phi$ given that $\mathfrak{m}$ is simply a single node $r^\mathfrak{m}$. 
 It now suffices to show $\ell^\mathfrak{m}(r^\mathfrak{m}) = \Phi \cap \Sigma$ assuming that $\bigwedge \ell^\mathfrak{m}(r^\mathfrak{m}) \in \Phi$. The inclusion $\subseteq$ holds trivially by definition of $\ell^\mathfrak{m}$ and the maximality of $\Phi$. For the reverse inclusion, let $\varphi \in \Phi \cap \Sigma$ and suppose towards a contradiction that $\varphi \notin \ell^\mathfrak{m}(r^\mathfrak{m})$. Since $\Sigma$ is closed under single negations, we have that $\neg \varphi \in \Sigma$. Therefore, since $\ell^\mathfrak{m}(r^\mathfrak{m})$ is a finite $\Sigma$-type, we obtain $\neg \varphi \in \ell^\mathfrak{m}(r^\mathfrak{m})$. From $\bigwedge \ell^\mathfrak{m}(r^\mathfrak{m}) \in \Phi$ it then follows that $\neg \varphi \in \Phi$, that is, $\varphi \notin \Phi$, which contradicts the assumption.
 
 For the inductive step, assume that $\mathrm{Sim}(\mathfrak{m}) \in \Phi$. Then, for any $r' \succ_v^\mathfrak{m} r^\mathfrak{m}$ we have that $\Diamond_v \mathrm{Sim}(\mathfrak{m}\!\upharpoonright_{r'}) \in \Phi$. By the witnessing Lemma \ref{WitnessCanonicalModel}, there is some $\Phi' \succ_v^{\rm c} \Phi$ such that $\mathrm{Sim}(\mathfrak{m}\!\upharpoonright_{r'}) \in \Phi'$. Therefore, by the inductive hypothesis there is some simulation $\sigma_{r'} : |\mathfrak{m}\!\upharpoonright_{r'}| \to W^{\rm c}$ for which $\mathfrak{m}\!\upharpoonright_{r'} \rightharpoonup \Phi'$. If we define $\sigma : |\mathfrak{m}| \to W^{\rm c}$ by setting $\sigma(r^{\mathfrak{m}}) := \Phi$ and $\sigma(w) := \sigma_{r'}(w)$ for $w$ reachable from some $r' \succ_v^\mathfrak{m} r^\mathfrak{m}$, it is easy to conclude that $\sigma : |\mathfrak{m}| \to W^{\rm c}$ is a simulation satisfying $\mathfrak{m} \rightharpoonup \Phi$; details are left to the reader. \qedhere
\end{enumerate}
\end{proof}

With these formulas at hand, we are ready to prove a version of Lemma \ref{lemExistsQuasiK4} for $[\mathsf{GL},\mathsf{GL}]^{\sf e}$, the last ingredient in our proof.

\begin{lem}\label{lemExistsQuasi}
Let $\Sigma\Subset\langfull$ and let $\mathcal{M}^{\rm c}$ be the canonical model of the expanding commutator $\sf [GL, GL]^{e}$. 
If $\Phi \in W^{\rm c}$, then there exists a finite $\Sigma$-quasimodel $\mathfrak{Q}$ and a $\mathfrak{Q}$-morphism $\pi^\mathfrak{Q}$ for $\sf [GL, GL]^{e}$ such that $\pi^\mathfrak{Q}(r^\mathfrak{Q}) = \Phi$.
\end{lem}
\begin{proof} 
The proof is almost identical to that of Lemma~\ref{lemExistsQuasiK4}:
we define a sequence $\{(\mathfrak{Q_i}, \pi^{i})\}_{i\in \mathbb{N}}$ where each $\mathfrak{Q}_i$ is a weak $\Sigma$-quasimodel and $\pi^{i}$ is a $\mathfrak{Q}_i$-morphism for $\sf [GL,GL]^{e}$. As before, $\mathfrak Q_0$ is a weak $\Sigma$-quasimodel consisting of a single point $r^0$ such that $\pi^0(r^0):=\Phi$ and $\mu^0(r^0)$ is chosen to satisfy $\mu^0(r^0) \rightharpoonup \Phi$ according to Lemma~\ref{itMomWitNow}.

The key difference is in the inductive step.
Assume that $\mathfrak Q_i$ and $\pi^{i}$ have been defined.
For every horizontal defect $\delta$ of some $\mu^{i}(t)$ with $t \in T^{i}$ that is not realized by any moment of $\mathfrak{Q}_i$, add a fresh successor $s \succ_h^{i+1} t$.
Choose $\Phi' \succ_h^{\rm c} \pi^{i}(t)$ and a $\Sigma$-moment $\mathfrak m'$ as provided by Lemma~\ref{itMomWitLater}.
In the proof of Lemma~\ref{lemExistsQuasiK4}, we would have set $\pi^{i+1}(s) = \Phi$; however, this could lead to an ill-founded quasimodel.

This is where the $\sf GL$ axiom and our simulation formulas come into play.
By Theorem~\ref{thmYF}, we have ${\rm Sim}(\mathfrak m')\in \Phi'$. Hence, by the definition of $\prec_h^{\rm c}$, it follows that $\hd\,{\rm Sim}(\mathfrak m')\in \pi^{i}(t)$.
Using the canonical witness property (Lemma~\ref{WitnessCanonicalModel}), there exists $\Psi' \,\rh^{\rm c}\, \pi^{i}(t)$ such that
\[
{\rm Sim}(\mathfrak m')\wedge \neg \hd\,{\rm Sim}(\mathfrak m')\in \Psi'.
\]
Set $\pi^{i+1}(s):=\Psi'$ and $\mu^{i+1}(s) := \mathfrak{m}'$. Then again by Theorem~\ref{thmYF}, we have $\mathfrak m' \rightharpoonup \Psi'$.

We claim that the sequence stabilizes at some $i \in \mathbb{N}$ satisfying the statement, that is, $\mathfrak Q_i$ is a finite $\Sigma$-quasimodel where the horizontal defects of its moments are all realized. Towards a contradiction, suppose not. Then, the sequence $\{(\mathfrak Q_i, \pi^{i})\}_{i \in \mathbb{N}}$ is increasing, with $\mathfrak Q_i \subseteq \mathfrak Q_{i+1}$, and uniformly finitely bounded. Let $\mathfrak Q_\infty := \bigcup_{i\in\mathbb N} \mathfrak Q_i$ and $\pi^\infty := \bigcup_{i\in\mathbb N} \pi^{i}$. Then, $\mathfrak{Q}^\infty$ defines an infinite $\Sigma$-quasimodel.
By K\"onig's lemma, there exists an infinite branch $t_0 \hr^\infty t_1 \hr^\infty t_2 \hr^\infty \cdots$ in $\mathfrak Q_\infty$. By Kruskal’s Theorem~\ref{thmKruskal}, there exist $i < j$ such that $\mu^\infty(t_i) \sqsubseteq^\Sigma \mu^\infty(t_j)$. By Theorem~\ref{thmYF}, it follows that ${\rm Sim}(\mu^\infty(t_i))\in \pi^\infty(t_j)$, contradicting $\neg \hd\,{\rm Sim}(\mu^\infty({t_i}))\in \pi^\infty(t_i)$ which is given by construction.
\end{proof}

\begin{thm} \label{Completeness}
The logic $[\sf{GL},\sf{GL}]^{\sf e}$ is sound and complete with respect to the following classes:
\begin{enumerate}
\item the class of (finite) expanding domain frames $(\sf{GL}\times \sf{GL})^{\sf e}$;
\item the class of (finite) embedding domain frames $(\sf{GL}\rtimes \sf{GL})^{\sf e}$;
\item the class of (finite) forward-confluent frames $\sf{GL}\rtimes \sf{GL}$.
\end{enumerate}
\end{thm}
\begin{proof}
Analogous to the proof of Theorem~\ref{Completeness$[K4,K4]^e$}, except that we now use Lemma \ref{lemExistsQuasi}.
\end{proof}

Although already proven by Gabelaia et al.~\cite{pml}, it is instructive to observe that we obtain an alternative proof of decidability.

\begin{thm}
$[\sf{GL},\sf{GL}]^{\sf e}$ is decidable.  
\end{thm}
\begin{proof}
Since the logic $[\mathsf{GL},\mathsf{GL}]^{\mathsf e}$ is (finitely) axiomatizable and enjoys the finite model property with respect to the classes of models based on each of the frames specified in Theorem~\ref{Completeness}, it follows that it is decidable. For more details, see, e.g., \cite[Theorem~6.13]{BRV01}. 
\end{proof}

\section{Concluding remarks} \label{sec:Concludingremarks}

We have shown that $[{\sf GL},{\sf GL}]^{\sf e}$ is complete for its expanding domain semantics, and that the logic remains complete if the horizontal component is replaced by $\sf K4$ but not if it is replaced by e.g.~$\sf Grz$, $\sf GL.3$ or $\sf K4.3$. This provides a partial solution to the problem posed by Gabelaia et al.~\cite{pml} of determining which expanding commutators are complete when one of the components is Noetherian.

We could likewise vary the vertical component and consider other logics of the form $[{\sf GL},{\sf L}_v]^{\sf e}$. Perhaps counter-intuitively, our conjecture is that in this case, setting ${\sf L}_v \in \{\mathsf{K4},\mathsf{K4.3}\}$ will lead to incompleteness, but setting ${\sf L}_v\in \{{\sf Grz},{\sf GL.3},{\sf Grz.3}\}$ to completeness. The reason for this is that the formulas ${\rm Sim}(\mathfrak m)$ are available over the class of $\sf Grz$ models~\cite{Fernandez11} but usually not for non-Noetherian logics, so our proof strategy could be adapted only in the former case.\footnote{Interestingly, they cannot be defined either over the class of all Noetherian models that are not assumed globally reflexive {\em or} globally irreflexive, i.e., $\sf wGrz$ models, so Noetherianness is not sufficient.} Logics with the $\sf 3$ axiom may need some extra work to retain linearity, but a related logic where the `vertical' component is linear has been treated by Aguilera et al.~\cite{AguileraDFM25} and similar techniques should work. As for incompleteness, the fact that our proof would fail does not prove that the logics are incomplete, but it would mirror the case of the closely-related {\em dynamic topological logic,} which is complete when the `vertical' logic is $\sf GL$~\cite{FDMontacute} but not when it is $\sf S4$~\cite{FDNonFin}. It is possible that e.g.~$[{\sf GL},{\sf S4}]^{\sf e}$ will be incomplete for similar reasons. This leads to two related questions.

\begin{question}\label{quesComp}
Which of the remaining logics proven decidable by Gabelaia et al.~\cite{pml} coincide with their expanding commutator?
\end{question}

\begin{question}
In those cases in which the commutator is incomplete, is the expanding domain logic finitely axiomatisable?
Can a natural axiomatisation be provided?
\end{question}

Note that `natural' here is not meant to have a technical meaning, e.g.~Peano arithmetic is not finitely axiomatisable but one may argue that the axioms are natural.

Another very intriguing line of inquiry is to consider intuitionistic companions of expanding products. We can define a Gödel-Tarski embedding $t$ of the language of intuitionistic modal logic with primitives $\bot,\wedge,\vee,\Rightarrow,\Diamond,\Box$ into the bi-modal language by letting $t$ commute with Booleans and setting
\begin{multicols}2

\begin{itemize}
    \item $t(p) = \boxdot_h p $
    \item $t(\varphi\Rightarrow\psi ) = \boxdot_h (t(\varphi)\to t(\psi)) $
    \item $t(\Diamond \varphi ) = \boxdot_h \vd  t(\varphi) $ 
    \item $t(\Box \varphi ) = \boxdot_h \vbx  t(\varphi) $.
\end{itemize}

\end{multicols}
\noindent 
If $\sf L$ is a bi-modal logic, the set of formulas whose translations are derivable is a well-defined logic, which we denote ${\sf I}(\sf L)$. We may also write ${\sf I}(\mathcal C)$ instead of ${\sf I}(\sf L)$ if $\sf L$ is the logic of the class $\mathcal C$.
As long as $\mathsf L_h$ extends ${\sf K4}$, ${\sf I}(\sf L)$ validates every intuitionistic tautology. Let us make the further convention that when ${\sf L}$ is the expanding commutator of some unimodal logic ${\sf L}_0$ in both vertical and horizontal components, then we write ${\sf IL}_0^{\sf e} $ instead of ${\sf I}({\sf L})$.

Then, $\sf IGL$~\cite{IGL} is ${\sf I}(\mathcal C)$, where $\mathcal C$ is the class of all ${\sf K4} \rtimes {\sf K4}$ frames where $ \peq_h;\vr $ is conversely well-founded. It should be clear that this class extends ${\sf GL} \rtimes {\sf GL}$, and in fact ${\sf IGL}\subsetneq {\sf IGL^e}$, since the formula $\neg\neg \Box (p\vee\neg p)$ belongs only to the right-hand side. To see that it belongs to ${\sf IGL^e}$, just note that $\Box (p\vee\neg p)$ holds on any $\hr$-maximal world. Meanwhile, the model $\mathfrak{M}$ (see \Cref{fig3}) given by $W:=\{0,1,2\}$, $0\vr 1$, $0\vr 2$, $1\hr 2$ and $\val p:= \{2\}$ provides an $\sf IGL$ counter-model.

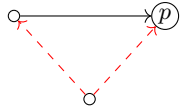
\begin{figure}[h] 
\centering
\small
\begin{tikzpicture}
	\begin{pgfonlayer}{nodelayer}
		\node [style=sroot, label=left:{}] (1) at (-1, 3.7) {};
		\node [style=sroot, label=above:{}] (2) at (1, 3.7) {$p$};
        \node [style=sroot, label=left:{}] (3) at (0, 2.6) {};

        \draw [style=edgeto] (1) to (2);
		\draw [style=edgeto, red, dashed] (3) to (1);
        \draw [style=edgeto, red, dashed] (3) to (2);
	\end{pgfonlayer}
\end{tikzpicture}
\caption{The model $\mathfrak{M}$.}
\label{fig3}
\end{figure}

The Noetherianness of $\hr$ can be reflected more explicitly in the logic if we enrich the intuitionistic language with an additional modality $\triangle$, interpreted simply as $\hbx$. Such a modality has been studied by Kuznetsov and Muravitsky, who showed that the Noetherianness of $\hr$ can be characterised by the validity of $(\triangle p\to p)\to p$~\cite{KuznetsovM86}.
Let us denote the set of formulas of this extended language valid over ${\sf GL} \rtimes {\sf GL}$ by $\sf KMGL$.
Then, $\sf KMGL$ is a (non-conservative) extension of $\sf IGL$ which embeds into a finitely axiomatisable bi-modal logic with the finite model property. However, we do not know of a `native' axiomatisation in the intuitionistic language for $\sf KMGL$.

We similarly obtain an additional family of logics by reversing the roles of the vertical and horizontal modalities in the translation, so that e.g.~$t(\varphi\Rightarrow\psi ) = \boxdot_v (t(\varphi)\to t(\psi)) $. Let us denote the logic obtained from $ [{\sf GL},{\sf GL}]^{\sf e} $ in this way by $\sf GLI$ and its Kuznetsov-Muravitsky extension by $\sf GLKM$. Following Balbiani et al.~\cite{BalbianiDF21}, we can check that $\sf IGL^e$ and $\sf GLI$ are incomparable since $(\Diamond p\to \Box q)\to \Box (p\to q)$ belongs only to $\sf IGL^e$ but $\Box(p\vee q)\to \Box p\vee\Diamond q $ only to $\sf GLI$. The decidability of $\sf GLI$ can be obtained from that of $ [{\sf GL},{\sf GL}]^{\sf e} $, but it is likely that it can also be obtained by more elementary means as in~\cite{BalbianiDF21}. This raises our two final questions.

\begin{question}
Is any of $\sf IGL^e$, $\sf KMGL$, $\sf GLI$, or $\sf GLKM$ finitely axiomatisable?    
\end{question}

\begin{question}
Is any of $\sf IGL^e$, $\sf KMGL$, $\sf GLI$, or $\sf GLKM$ decidable in primitive recursive time? 
\end{question}

\bibliographystyle{plain}
\bibliography{biblio}

\end{document}